\documentclass{imanum}
\usepackage{graphicx}
\usepackage{color}
\usepackage{mathtools}
\usepackage{tikz}
\usepackage{pgfplots}
\mathtoolsset{showonlyrefs}

\newcommand{\Om}{\Omega}
\newcommand{\Omt}{\Omega(t)}
\newcommand{\Omh}{\Omega_h}
\newcommand{\Omht}{\Omega_h(t)}
\newcommand{\Omhx}{\Omega_h(\bfx)}
\newcommand{\Ga}{\Gamma}
\newcommand{\Gat}{\Gamma(t)}
\newcommand{\Gah}{\Gamma_h}
\newcommand{\Gahx}{\Gamma_h(\bfx)}

\newcommand{\tr}{\mathbf{\gamma}}

\newcommand{\T}{^\mathrm{T}}

\newcommand{\Sohx}{S_{0,h}(\bfx)}
\newcommand{\Sohxt}{S_{0,h}(\bfx(t))}

\newcommand{\bfd}{\mathbf{d}}
\newcommand{\bfe}{\mathbf{e}}
\newcommand{\bff}{\mathbf{f}}

\newcommand{\bfu}{\mathbf{u}}
\newcommand{\bfv}{\mathbf{v}}
\newcommand{\bfw}{\mathbf{w}}
\newcommand{\bfx}{\mathbf{x}}
\newcommand{\bfy}{\mathbf{y}}
\newcommand{\bfz}{\mathbf{z}}
\newcommand{\bfM}{\mathbf{M}}
\newcommand{\bfA}{\mathbf{A}}
\newcommand{\bfK}{\mathbf{K}}

\newcommand{\mat}{\partial^\bullet}
\newcommand{\nb}{\nabla}

\newcommand{\laplace}{\Delta}
\newcommand{\dfdt}[2]{\frac{\mathrm{d}#1}{\mathrm{d}#2}}
\newcommand{\dfdx}[2]{\frac{\partial #1}{\partial #2}}

\newcommand{\dx}{\mathrm{d}x}

\newcommand{\dTheta}{\mathrm{d}\theta}

\newcommand{\R}{\mathbb{R}}

\newcommand{\calTh}{\mathcal{T}_h}

\newcommand{\hatx}{\widehat{x}}
\newcommand{\hatT}{\widehat{T}}

\newcommand{\brev}{\color{black}}
\newcommand{\erev}{\color{black}}

\begin{document}

\title{Finite element analysis for a diffusion equation \\
	on a harmonically evolving domain}

\shorttitle{Finite elements on a harmonically evolving domain}

\author{
{\sc Dominik Edelmann\thanks{Email: edelmann@na.uni-tuebingen.de}}\\[2pt]
Mathematisches Institut, Universit\"at T\"ubingen\\
Auf der Morgenstelle 10, 72076 T\"ubingen, Germany
}
\shortauthorlist{D. Edelmann}

\maketitle

\begin{abstract}
	{We study convergence of the evolving finite element semi-discretization of a parabolic partial differential equation on an evolving bulk domain. The boundary of the domain evolves with a given velocity, which is then extended to the bulk by solving a Poisson equation. The numerical solution to the parabolic equation depends on the numerical evolution of the bulk, which yields the time-dependent mesh for the finite element method. The stability analysis works with the matrix--vector formulation of the semi-discretization only and does not require geometric arguments, which are then required in the proof of consistency estimates. We present various numerical experiments that illustrate the proven convergence rates.}
	{evolving domain, harmonic velocity, diffusion, evolving finite elements, error analysis} \\
	Classification...
\end{abstract}
\renewcommand{\nu}{\mathrm{n}}
\section{Introduction}\label{section: introduction}
This paper studies the numerical discretization of a diffusion equation in a time-dependent domain that is specified by the velocity of its boundary. The interior velocity is determined as the solution of a Laplace equation with the given boundary velocity as Dirichlet data.

The strong formulation of this model is to find the time-dependent domain $\Omt \subset \R^n$ $(n=2, 3)$, $t \in [0,T]$, which moves with a velocity $v$ that is the harmonic extension of the a priori given velocity $v^\Ga$ of the boundary $\Gat = \partial \Omt$.  That is, $v$ is not given explicitly but determined as the solution of the Laplace equation, for all $t \in [0,T]$,
\begin{subequations}\label{eq: velocity Laplace equation}
	\noeqref{subeq: harmonic ext,subeq: surface vel}
	\begin{alignat}{3}
		- \laplace v(x,t) &= 0\,, \qquad &&x \in \Omt \,,\label{subeq: harmonic ext} \\
		v(x,t) &= v^\Ga(x,t) \,,\qquad &&x \in \Gat \label{subeq: surface vel} \,.
	\end{alignat}
\end{subequations}

In $\Omt$ we seek a solution $u=u(x,t)$  with given initial data $u(\cdot,0)=u_0$ to the partial differential equation 
\begin{align}\label{eq: strong formulation}
	\mat u(x,t) + u(x,t) \nb \cdot v(x,t) - \beta \laplace u(x,t) = f(x,t)\,,\quad x \in \Omt \,,\quad t \in [0,T] \,,
\end{align}
where $\mat$ denotes the material derivative, $\nb \cdot v$ is the divergence of the velocity and $\beta > 0$ is a given diffusion coefficient. On the boundary, we impose the Neumann condition $\dfdx{u}{\nu}(x,t) = g(x,t)$, $x \in \Gat$, $t \in [0,T]$, where $\nu$ denotes the unit outward pointing normal to $\Gat$.

Convection--diffusion equations in time-dependent domains have gained considerable interest in the past decades. A model similar to \eqref{eq: strong formulation} (without \eqref{eq: velocity Laplace equation}) with an additional convection term, together with homogeneous Dirichlet boundary conditions is analyzed in \cite{BC06} and \cite{BG04} in an arbitrary Lagrangian--Eulerian (ALE) framework, where the velocity is determined by the ALE mapping. In \cite{BG04}, the velocity of the boundary is prescribed and the ALE mapping is constructed as the harmonic extension of the boundary positions. This approach is first proposed in \cite{NF99} in the context of a generic conservation law on a moving domain, see also \cite{Gastaldi01,FN04} and the references therein.

Diffusion equations on evolving surfaces are analyzed in \cite{DE07a,DE07b,DE13a} for a given velocity, and there are recent works, where the velocity is not given explicitly but determined by various velocity laws that depend on the solution of the diffusion equation on the surface, see \cite{KLLP17,KL18,KLL19}.

\brev
Through the numerical analysis of the problem with a given boundary velocity  \eqref{eq: velocity Laplace equation}--\eqref{eq: strong formulation} we will develop techniques which are expected to be essential for more involved problems, such as the tumor growth model of \cite{EKS19}, where a  bulk--surface model for tissue growth is presented, together with a numerical algorithm. Instead of the coupled system \eqref{eq: velocity Laplace equation}--\eqref{eq: strong formulation} in \cite[(1.1)--(1.3) \& Section 6.1.2]{EKS19} they consider the boundary velocity $v^\Ga$ given by the forced mean curvature flow
\begin{align}
	v^\Ga &= \frac{u}{\alpha} + \beta H \,,
\end{align}
and instead of \eqref{eq: strong formulation} they consider an elliptic boundary value problem in the moving bulk. Here $H$ denotes the mean curvature of the boundary surface and $\alpha$, $\beta$ are given positive constants.
\erev

In this paper, we 
prove error bounds for the spatial semi-discretization of the coupled problem \eqref{eq: velocity Laplace equation}--\eqref{eq: strong formulation} with isoparametric finite elements of polynomial degree at least two. More precisely, we show $H^1$-norm error bounds in the positions and the velocity $v$ that are uniform in time, and $L^\infty L^2$-norm and $L^2 H^1$-norm error bounds for the solution $u$ of the diffusion equation. The proof clearly separates the stability and consistency analysis. To prove stability of the semi-discrete equations, we adapt techniques recently used in \cite{KLLP17,Kov17} to the present situation.  \brev The stability analysis of the semi-discrete problems uses energy estimates. Transport formulae are used to relate mass and stiffness matrices corresponding to different \brev discrete domains. In order to estimate errors between these matrices on different domains, \erev a key issue is to control the $W^{1,\infty}$-norm of the position error uniformly in time. This is done with an inverse estimate, that yields an $\mathcal{O}(h^{k-n/2})$ bound uniformly in time, which is small only for $k \ge 2$. The stability analysis of the semi-discrete diffusion equation uses the same techniques and is based on the stability analysis of the semi-discrete velocity law. Moreover, it becomes clear how the position error affects the error in the numerical solution to \eqref{eq: strong formulation}.

The stability analysis relies on smallness assumptions on the defects. These are shown to be true in the following consistency analysis, that uses geometric approximation estimates and interpolation results. The final convergence result is then obtained by combining stability and consistency estimates together with interpolation error bounds.
%
%
%

The paper is organized as follows.

In Section~\ref{section: problem formulation} we recall basic notation and formulate a diffusion equation on an evolving domain together with the above velocity law. We derive the weak formulation.

In Section~\ref{section: EBFEM} we describe the high-order evolving  finite element approximation of the problem. After introducing an \emph{exact triangulation} of the curved domain, we define the computational domain and the finite element method. We describe the spatial semi-discretization and derive a matrix--vector formulation, which will be crucial for the stability analysis.

In Section~\ref{section: main result} we state the main result of the paper, which gives convergence estimates for the spatial semi-discretization with evolving isoparametric finite elements of polynomial degree at least 2. We outline the main ideas of the proof.

In Section~\ref{section: auxiliary results} we collect auxiliary results that will be needed for the following analysis. The first part deals with the evolving mass and stiffness matrices and their properties, which are crucial in the stability analysis. The second part collects geometric estimates which will be needed only for consistency analysis.

Section~\ref{section: stability velocity} analyses the stability of the semi-discrete velocity law without a diffusion equation on the evolving domain. In Section~\ref{section: stability diffusion}, we extend the stability analysis to the semi-discrete diffusion equation. Section~\ref{section: defect bounds} contains the consistency analysis, that is, estimating the defects obtained on inserting the interpolated exact solutions into the numerical scheme.

In Section~\ref{section: proof} we prove the main convergence result by combining the stability and consistency estimates. Section~\ref{section: examples} provides several numerical experiments which illustrate the theoretical results. 

\section{Problem formulation}\label{section: problem formulation}
\subsection{Basic notation}
For $t \in [0,T]$, let $\Omt \subseteq \R^n$ ($n=2, 3$) be an open, bounded and connected set with smooth boundary $\Gat = \partial \Omt$ and $\Om_0 = \Om(0)$, $\Ga_0 = \Ga(0)$. We denote $\overline{\Omt} = \Omt \cup \Gat$. We assume that there exists a sufficiently smooth map $X: \Om_0 \cup \Ga_0 \times [0,T] \to \R^n$ such that
\[ \Omt=\{ X(p,t):\, p \in \Om_0 \} \,,\quad \Gat = \{ X(p,t):\, p \in \Ga_0  \} \,. \]
The velocity $v(x,t)$ at a point $x=X(p,t) \in \overline{\Omt}$ is defined by
\[ v(X(p,t),t) = \dfdx{}{t} X(p,t) \,. \]
For a function $u=u(x,t)$, $x \in \overline{\Omt}$, $t \in [0,T]$, the material derivative at $x=X(p,t)$ is defined by
\begin{align}\label{eq: material derivative}
	\mat u(x,t) = \dfdt{}{t} u(X(p,t),t) = \dfdx{}{t} u(x,t) + \nb u(x,t) \cdot v(x,t) \,. 
\end{align}
For $x \in \Gat$, we denote by $\nu=\nu(x,t)$ the unit outward pointing normal to $\Gat$. We define the space--time domain $\Om_T$ and the space--time surface $\Ga_T$ by
\begin{align}\label{space-time domain}
	\Om_T = \bigcup_{t \in [0,T]} \Omt \times \{t\} \,,\quad \Ga_T = \bigcup_{t \in [0,T]} \Gat \times \{t\} \,.
\end{align}
For functions $\varphi, \psi$ defined on $\Omt$, we have bilinear forms
\begin{align}\label{eq: bilinear forms}
	m(\varphi,\psi) &= \int_{\Omt} \varphi \psi \dx \,,\\
	a(\varphi,\psi) &= \int_{\Omt} \nb \varphi \cdot \nb \psi \dx \,.
\end{align}
Note that these bilinear forms explicitly depend on $t$, but we will omit the argument $t$, for brevity. It will always be clear from context for which $t \in [0,T]$ the bilinear forms are evaluated.

\subsection{Diffusion equation}
We assume that $u=u(\cdot,t)$ is the density of a scalar quantity on $\Omt$ (for example, mass per unit volume). We follow a construction of \cite[Section 3]{DE07a} to obtain a diffusion equation with Neumann boundary conditions:
\begin{equation}\label{eq: conservation and diffusion law}
	\left\{
	\begin{aligned}\
		\mat u + u \nb \cdot v - \beta \laplace u &= f \quad&&\text{in }&&\Omt \,,\\
		\dfdx{u}{\nu} = \nb u \cdot \nu &= g \quad&&\text{on }&&\Gat \,,
	\end{aligned}
	\right.
\end{equation}
where $\nb \cdot v$ denotes the divergence of the velocity, $\beta > 0$ is a given diffusion coefficient and $\nu$ the unit outward pointing normal to $\Gat$. 

\subsection{Harmonic velocity law}
Contrary to existing works (cf. \cite{ER17pre}), the velocity $v(\cdot,t)$ of $\Omt$ is not given explicitly. Instead, only the velocity of the boundary $\Gat = \partial \Omt$ is given; the velocity of the bulk is then determined as the harmonic extension, i.e. as the solution to the Laplace equation. More precisely, we have the following differential equation for $v(x,t)$: for each $t \in [0,T]$
\begin{equation}\label{eq: Poisson equation velocity}
	\left\{
	\begin{aligned}
		- \Delta v(\cdot,t) &= 0 \quad &&\text{in}\,&&\Omt \,,\\
		v(\cdot,t) &= v^\Ga(\cdot,t) \quad &&\text{on}\,&&\Gat \,.
	\end{aligned}
	\right.
\end{equation}
\brev We assume that $v^\Ga$ is defined on a neighborhood of $\Ga_T$, as defined in \eqref{space-time domain}. \erev This system is considered together with the position ODEs: for each $p \in \overline{\Om(0)}$
\begin{equation}\label{eq: position ODE}
	\left\{
	\begin{aligned}
		\dfdt{}{t}X(p,t) &= v(X(p,t),t) \,,\\
		X(p,0) &= p \,.
	\end{aligned}
	\right.
\end{equation}
We consider an equivalent problem with homogeneous Dirichlet boundary conditions: assume that $v^\Ga(\cdot,t)$ is the trace of a given function \brev $w(\cdot,t) \in H^1(\Omt)^n$ and consider the equivalent problem: find $\widetilde{v}(\cdot,t) \in H_0^1(\Omt)^n$ \erev such that
\begin{equation}\label{eq: Poisson equation velocity homogeneous}
	\left\{
	\begin{aligned}
		-\Delta \widetilde{v}(\cdot,t) &= \Delta w(\cdot,t) \quad &&\text{in}\,&&\Omt \,,\\
		\widetilde{v}(\cdot,t) &= 0 \quad &&\text{on}\,&&\Gat \,.
	\end{aligned}
	\right.
\end{equation}
It is easily seen that the solution $v = \widetilde{v} + w$ to \eqref{eq: Poisson equation velocity} does not depend on the choice of $w$.

\subsection{Coupled problem: strong and weak formulation}
We consider the following system of partial differential equations: \brev for given $\beta>0$, $f: \R^n \times [0,T] \to \R$, $g: \R^n \times [0,T] \to \R$ and $v^\Ga: \R^n \times [0,T] \to \R^n$, find the unknown function $u: \Om_T \to \R$, the unknown velocity field $v: \Om_T \to \R^n$ and the unknown position function $X: \Om_0 \cup \Ga_0 \times [0,T] \to \R^n$ \erev such that for all $t \in [0,T]$
\begin{equation}\label{eq: conservation diffusion coupled to semi-explicit velocity law}
	\left\{
	\begin{aligned} 
		\mat u(\cdot,t)+ u(\cdot,t) \nb \cdot v(\cdot,t) - \beta \laplace u (\cdot,t) &= f(\cdot,t) \quad&&\text{in }&&\Omt \,, \\
		\dfdx{u}{\nu} (\cdot,t) &= g(\cdot,t) \quad&&\text{on }&&\Gat \,, \\
		\dfdt{X}{t}(\cdot,t)&= v(X(\cdot,t),t) \quad&&\text{in }&&\Om_0\cup \Ga_0 \,, \\
		-\laplace v(\cdot,t) &= 0 \quad&&\text{in }&&\Omt \,, \\
		v(\cdot,t) &= v^\Ga(\cdot,t) \quad&&\text{on }&&\Gat \,.
	\end{aligned}
	\right.
\end{equation}
Without loss of generality, we assume $\beta=1$ and $g \equiv 0$ in the following.

\brev
\begin{remark}\label{remark: domain motion independent}
	The last three equations of \eqref{eq: conservation diffusion coupled to semi-explicit velocity law} purely describe the motion of the domain $\Omt$ and are independent of the parabolic equation for $u$. The latter includes the velocity $v$ in the material derivative as well as the divergence of the velocity in the equation. This is reflected in the stability analysis, which is first done for the discretization of the domain motion, and then extended to the parabolic equation. On the other hand, if the velocity field $v$ was given for the whole domain, the finite element analysis for the parabolic equation alone would be remarkably easier. Convergence results for these types of problems with given velocity on the whole domain are found in \cite[Section 7]{ER17pre}.
\end{remark}
\erev

We now derive a weak formulation. By multiplying the first equation with an arbitrary test function $\varphi \in H^1(\Omt)$ such that $\mat \varphi$ exists in $L^2(\Omt)$, integrating over $\Omt$, using the Leibniz formula, Green's formula and the Neumann boundary condition, we arrive at
\begin{align}
	\dfdt{}{t} \int_{\Omt} u \varphi + \int_{\Omt} \nb u \cdot \nb \varphi = \int_{\Omt} f \varphi  + \int_{\Omt} u \mat \varphi \,.
\end{align}
Multiplying \eqref{eq: Poisson equation velocity homogeneous} with arbitrary test function \brev $\psi \in H_0^1(\Omt)^n$ \erev, integrating over $\Omt$ and using Green's formula, we obtain
\begin{align*}
	\int_{\Omt} \nb \widetilde{v} \cdot \nb \psi = - \int_{\Omt} \nb w \cdot \nb \varphi  \,,
\end{align*}
where we have used that the boundary integrals vanish thanks to \brev $\psi \in H_0^1(\Omt)^n$ \erev. Here, the dot denotes the Euclidean inner product of the vectorizations of the matrices, i.e. the Frobenius norm inner product of the matrices. Again, it can be shown that the weak solution $v = \widetilde{v} + w$ does not depend on $w$.

The weak formulation of the diffusion equation and the domain evolution thus reads: \brev find $u(\cdot,t) \in H^1(\Omt)$, $\widetilde{v}(\cdot,t) \in H_0^1(\Omt)^n$ such that for all $\varphi \in H^1(\Omt)$ with $\mat \varphi \in L^2(\Omt)$, $\psi \in H_0^1(\Omt)^n$ \erev and all $t \in [0,T]$
\begin{align}\label{eq: weak formulation}
	\begin{aligned}
		\dfdt{}{t} \int_{\Omt} u \varphi + \int_{\Omt} \nb u \cdot \nb \varphi &= \int_{\Omt} f \varphi + \int_{\Omt} u \mat \varphi \,, \\
		\int_{\Omt} \nb \widetilde{v} \cdot \nb \psi &= -\int_{\Omt} \nb w \cdot \nb \psi \,, \\
		v &= \widetilde{v} + w \,, \\
		\dfdt{X}{t} &= v \,.
	\end{aligned}
\end{align}
This is considered together with given initial data $u(\cdot,t)=u_0(\cdot)$, $X(\cdot,0)=Id$.

We assume throughout the paper that there exists a unique weak solution with sufficiently high Sobolev regularity on $[0,T]$. \brev Precise regularity assumptions will be given in Theorem \ref{theorem: error bound}. \erev

\brev From now on, we will be working in the technically more challenging three-dimensional case. All of the upcoming results are valid in the two-dimensional case as well. \erev

\section{Evolving bulk finite elements}\label{section: EBFEM}
In this section we briefly recall the evolving isoparametric finite element method which is used for semi-dis\-cre\-ti\-zation in space. We refer to \cite{ER13,ER17pre} for a more detailed introduction \brev into the construction of isoparametric finite elements \erev.


In the following, we denote $\Om_0=\Om(0)$ for brevity. The initial domain $\Om_0$ is triangulated and the nodes are then evolved in time by solving the position ODE $\dot{x}_i = v(x_i,t)$ in each node, together with \eqref{eq: Poisson equation velocity}.

\subsection{High-order domain approximation}\label{subsec: high order domain approximation}
We construct a triangulation $\calTh^{(1)}$ of $\Om_0$ consisting of closed simplices with maximal diameter $h$. The union of all simplices of $\calTh^{(1)}$ defines a polyhedral approximation $\Omh$ of $\Om_0$, whose boundary $\Gah = \partial \Omh$ is an interpolation of $\Ga_0$.

Each simplex $T \in \calTh^{(1)}$ corresponds to a curved simplex $T^\mathrm{c} \subset \Om$, which is parametrized over the unit simplex $\widehat{T}$ with a map $\Phi_T^\mathrm{c} = \Phi_T + \rho_T$. Here, $\Phi_T$ denotes the usual affine function that maps $\widehat{T}$ onto $T$. For the construction of an appropriate $\rho_T$, we refer to \cite{ER13}. The union of those curved simplices can be considered as an \emph{exact triangulation} of $\Om_0$. Using the map $\Phi_T^\mathrm{c}$, we can define an isoparametric mapping $\Phi_T^{(k)}$, that maps the unit simplex $\widehat{T}$ to a polynomial simplex $T^{(k)}$. $\Omh^{(k)}$ is then defined as the union of elements in $\calTh^{(k)}$, where
\[ \calTh^{(k)} := \{ T^{(k)}:T \in \calTh^{(1)} \} \,,\quad T^{(k)} := \{ \Phi_T^{(k)}(\hatx): \hatx \in \hatT \} \,. \]

\subsection{Evolving finite element method}
Here and in the following, we use the notational convention that vectors and matrices are denoted with bold-face letters. \brev As mentioned, we set $n=3$ and assume that the order $k \ge 2$ is fixed. \erev

Based on the previous subsection, we obtain a triangulation of $\Om_0$, whose nodes $x_1^0,\ldots,x_N^0$ are collected in a vector $\bfx^0=(x_1^0,\ldots,x_N^0) \in \R^{3N}$. We assume that the enumeration is such that exactly the first $N_\Ga$ nodes lie on the boundary $\Ga_0 = \partial \Om_0$. The nodes are evolved in time and collected in a vector $\bfx(t) = (x_1(t),\ldots,x_N(t))$ with $\bfx(0)=\bfx^0$. We use the notation
\[ \bfx(t) = \begin{pmatrix} \bfx^\Ga(t) \\ \bfx^\Om(t) \end{pmatrix} \]
to indicate which nodes live on the boundary. The nodal vector $\bfx=\bfx(t)$ defines a computational domain $\Omh(\bfx)=\Omh(\bfx(t))$ with boundary $\Gahx$. The finite element basis functions $\varphi_j(\cdot,t): \Omh(\bfx(t)) \to \R$ satisfy
\[ \varphi_j(x_k(t),t) = \delta_{jk} \,,\quad 1 \le j,k \le N \,, \]
and their pullback to the reference triangle is polynomial of degree $k$. Note that, since the velocity is not given explicitly, we are in general not able to find the exact positions $x_j^*(t) = X(x_j^0,t)$, so that $\Omh(\bfx(t))$ is \emph{not} the triangulation of $\Om(t)$ corresponding to the exact positions $X(x_j^0,t)$. It is therefore more convenient to denote the dependence on $\bfx$ instead of $t$, i.e. to write $\Omh(\bfx)$ and not $\Omh(t)$, etc.

The finite element space is now given as
\[ S_h(\bfx) = \mathrm{span}\{ \varphi_1[\bfx],\ldots,\varphi_N[\bfx] \} \,, \]
where $\varphi_j[\bfx](\cdot) = \varphi_j(\cdot,t)$ for $\bfx=\bfx(t)$.

We use the notation
\begin{align*}
	\Sohx &= \{ \varphi_h[\bfx] \in S_h(\bfx): \, \tr_h \varphi_h[\bfx] = 0 \} = \mathrm{span} \{ \varphi_{N_\Ga+1}(\bfx),\ldots,\varphi_N[\bfx] \} \,,
\end{align*}
where $\tr_h \varphi_h$ denotes the trace of a function $\varphi_h$ defined on $\Omhx$ on $\Gahx$. We set
\[ X_h(p_h,t) = \sum_{j=1}^N x_j(t) \varphi_j[\bfx(0)](p_h)\,,\quad p_h \in \Omh^0\cup\Gah^0 \,,\]
which has the properties that $X_h(x_k^0,t) = x_k(t)$ and $X_h(x_j^0,0) = x_j^0$, implying that $X_h(x,0) = x$ for all $x \in \Omh^0 \cup \Gah^0$. 
The discrete velocity $v_h(x,t)$ at a particle $x = X_h(p_h,t)$ is given by
\[ v_h(X_h(p_h,t),t) = \dfdt{}{t} X_h(p_h,t) \,.\]
\brev The basis functions satisfy the transport property
\begin{align}\label{transport property}
	\dfdt{}{t} \left( \varphi_j[\bfx(t)](X_h(p_h,t)) \right) = 0 \,,
\end{align}
which implies $\varphi_j[\bfx(t)](X_h(p_h,t)) = \varphi_j[\bfx(0)](p_h)$. \erev For the discrete velocity, this means
\begin{align*}
	v_h(x,t) &= v_h(X_h(p_h,t),t) = \dfdt{}{t} \sum_{j=1}^N x_j(t) \varphi_j[\bfx(0)](p_h) =\sum_{j=1}^N v_j(t) \varphi_j[\bfx(t)](x) \text{ with }v_j = \dot{x}_j \,.
\end{align*}
In particular, $v_h(\cdot,t) \in S_h(\bfx(t))$. For a finite element function $u_h(x,t) = \sum_{j=1}^N u_j(t) \varphi_j[\bfx(t)](x)$, the discrete material derivative at $x=X_h(p_h,t)$ is defined by
\begin{align*}
	\mat_h u_h(x,t) = \dfdt{}{t} u_h(X_h(p_h,t),t) = \sum_{j=1}^N \dot{u}_j(t) \varphi_j[\bfx(t)](x) \,,
\end{align*}
where we have used the transport property again. In particular: $\mat_h u_h(\cdot,t) \in S_h(\bfx(t))$.

\subsection{Spatial semi-discretization and matrix--vector formulation}
The evolving finite element discretization of \eqref{eq: weak formulation} reads: find the unknown position vector $\bfx(t) \in \R^{3N}$ and the unknown finite element functions $u_h(\cdot,t) \in S_h(\bfx(t))$, $\widetilde{v}_h(\cdot,t) \in \Sohxt^3$ such that for all $\varphi_h(\cdot,t) \in S_h(\bfx(t))$ with $\mat_h \varphi_h \in S_h(\bfx(t))$ and all $\psi_h(\cdot,t) \in \Sohxt^3$
\begin{align}\label{eq: spatial semidiscretization}
	\begin{aligned}
		\dfdt{}{t} \int_{\Omh(\bfx(t))} u_h \varphi_h + \int_{\Omh(\bfx(t))} \nb u_h \cdot \nb \varphi_h &= \int_{\Omh(\bfx(t))} f \varphi_h + \int_{\Omh(\bfx(t))} u_h \mat_h \varphi_h \,, \\
		\int_{\Omh(\bfx(t))} \nb \widetilde{v}_h \cdot \nb \psi_h &= - \int_{\Omh(\bfx(t))} \nb w_h \cdot \nb \psi_h \,,
	\end{aligned}
\end{align}
together with 
\[ \dfdx{}{t} X_h(p_h,t) = v_h(X_h(p_h,t),t) \,,\quad X_h(p_h,0)=p_h \,, \]
for $p_h \in \Om_h^0 \cup \Gah^0$, where $v_h = \widetilde{v}_h + w_h$. The initial values for the nodal vector $\bfu$ corresponding to $u_h(\cdot,0)$ and the nodal vector $\bfx(0)$ are taken as the exact initial values of the nodes $x_j^0$ of the initial triangulation of $\Om_0$:
\begin{align}\label{eq: initial data discrete}
	u_j(0)=u(x_j^0,0) \,,\quad x_j(0)=x_j^0 \quad (j=1,\ldots,N) \,.
\end{align}

We now show that the nodal vectors $\bfu \in \R^{N}$ and $\bfv \in \R^{3N}$ corresponding to the finite element functions $u_h$ and $v_h$, respectively, together with the position vector $\bfx \in \R^{3N}$ satisfy a system of differential equations. We set (omitting the omnipresent argument $t$)
\[ u_h = \sum_{j=1}^N u_j \varphi_j[\bfx]\,,\quad v_h = \sum_{j=1}^N v_j \varphi_j[\bfx]  \]
with $u_j \in \R$, $v_j \in \R^3$ and collect the nodal values in vectors $\bfu \in \R^N$, $\bfv \in \R^{3N}$. The domain-dependent mass and stiffness matrices $\bfM(\bfx)$ and $\bfA(\bfx)$ are defined by
\begin{align}
	\bfM(\bfx)_{jk} &= \int_{\Omhx} \varphi_j[\bfx] \varphi_k[\bfx] \dx \,,\\
	\bfA(\bfx)_{jk} &= \int_{\Omhx} \nb \varphi_j[\bfx] \cdot \nb \varphi_k[\bfx] \dx \,.
\end{align} 
In view of the following discretization of the velocity law, we use the notation
\begin{align}\label{def: A split}
\bfA(\bfx) &= \begin{pmatrix}
\bfA_{11}(\bfx) & \bfA_{12}(\bfx) \\ \bfA_{21}(\bfx) & \bfA_{22} (\bfx)
\end{pmatrix} \,,
\end{align}
where $\bfA_{11}(x) \in \R^{N_\Ga \times N_\Ga}$ and $\bfA_{22}(\bfx) \in \R^{N_\Om\times N_\Om}$. $\bfA_{22}(\bfx)$ thus corresponds to the finite element functions which vanish on the boundary. We will use the same notation for $\bfM(\bfx)$ when it is necessary. It is important to note that the sub-matrix $\bfA_{22}(\bfx)$ is \brev invertible.\erev

For the right-hand side of the diffusion equation, we define the vector
\begin{align}
	\bff(\bfx(t))_k &= \int_{\Omhx} f \varphi_k[\bfx] \dx \,.
\end{align}
\brev By linearity, the transport property implies $\mat_h \varphi_h = 0$, so the first equation of  \eqref{eq: spatial semidiscretization} is equivalent to \erev
\[ \dfdt{}{t} \left( \bfM(\bfx(t)) \bfu(t) \right) + \bfA(\bfx(t)) \bfu(t) = \bff(\bfx(t))\,. \]

For the velocity law, we remind that the nodes $x_j(t)$, $j=1,\ldots,N_\Ga$, on the boundary are known explicitly since $v^\Ga(\cdot,t)$ is prescribed. Writing $v_j(t) = v^\Ga(x_j(t),t)$, we have the finite element interpolation of $v^\Ga$:
\[ v_h^\Ga(\cdot,t) = \sum_{j=1}^{N_\Ga} v_j(t) \varphi_j[\bfx(t)](\cdot) \,.  \]
We write
\[ w_h(\cdot,t) = \sum_{j=1}^N w_j(t) \varphi_j[\bfx(t)](\cdot) \,,\quad w_j(t) = v_j(t) \text{ for }j=1,\ldots,N_\Ga \,, \]
for an arbitrary extension of $v_h^\Ga$. \brev Noting that $\bfv$ has three components\erev, a short calculation shows that the second equation of \eqref{eq: spatial semidiscretization} is equivalent to
\begin{align}\label{eq: matrix vector formulation}
	\left(I_3 \otimes \bfA_{22}(\bfx) \right) \bfv^\Om = - \left( I_3 \otimes \begin{pmatrix}
	\bfA_{21}(\bfx) & \bfA_{22}(\bfx)
	\end{pmatrix} \right) \begin{pmatrix}
	\bfv^\Ga \\ \bfw^\Om
	\end{pmatrix} \,,
\end{align}
where $\bfw^\Om$ is the vector containing the nodal values of $w_h$ in the inner nodes. \brev Here, $I_3$ denotes the identity matrix of size $3 \times 3$ and $\otimes$ denotes the Kronecker product. \erev The solution $v_h$ we are seeking is then obtained by $v_h = v_h^\Ga + \widetilde{v}_h$ and corresponds to the nodal vector
\[ \bfv = \begin{pmatrix}
\bfv^\Ga \\ \bfv^\Om + \bfw^\Om
\end{pmatrix} \,, \quad \text{i.e. } v_h = \sum_{j=1}^N v_j \varphi_j[\bfx] \,. \]
Using the fact that $\bfA_{22}(\bfx)$ is invertible, it is easily seen that the solution $\bfv$ does not depend on the particular choice of $\bfw^\Om$, which is why we use $\bfw^\Om=0$.

The matrix--vector formulation reads (omitting the Kronecker product notation)
\begin{align}
\begin{aligned}
\dfdt{}{t} (\bfM(\bfx) \bfu) + \bfA(\bfx)\bfu &= \bff(\bfx) \,,\\
- \bfA_{22}(\bfx) \bfv^\Om(\bfx) &= \bfA_{21}(\bfx) \bfv^\Ga(\bfx) \,,\\
\dfdt{}{t} \begin{pmatrix} \bfx^\Ga \\ \bfx^\Om \end{pmatrix} = \dot{\bfx} &= \bfv = \begin{pmatrix} \bfv^\Ga \\ \bfv^\Om \end{pmatrix} \,.
\end{aligned}
\end{align}
The initial nodal vectors $\bfu(0)$ and $\bfx(0)$ are chosen as in \eqref{eq: initial data discrete}.

We will see in the following sections that the matrix--vector formulation is the only tool used in the stability analysis, where geometric estimates are only needed for the consistency analysis.

\subsection{Lifted finite element space}
In the error analysis, we compare functions on three different domains: the exact domain $\Omt$, the discrete domain $\Omht = \Omh(\bfx(t))$ obtained by the finite element method and the \emph{interpolated exact domain} $\Omh^*(t) = \Omh(\bfx^*(t))$, which is the computational domain corresponding to the nodal vector $\bfx^*(t)$ with the exact positions $x_j^*(t) = X(x_j^0,t)$ of the nodes at time $t$ \brev and only available in theory\erev.

Any finite element function $u_h \in S_h(\bfx)$ on the discrete computational domain, with nodal values $u_j$, $j=1,\ldots,N$, is related to a finite element function $\widehat{u}_h \in S_h(\bfx^*)$ with the same nodal values:
\[ \widehat{u}_h = \sum_{j=1}^N u_j \varphi_j[\bfx^*] \,. \]

Based on Section~\ref{subsec: high order domain approximation}, we obtain a map $\Lambda_h(\cdot,t): \Omh(\bfx^*(t)) \to \Omt$ (cf. \cite{ER13, ER17pre}), \brev that is defined element-wise and maps the curved elements of the triangulation of $\Omh(\bfx^*(t))$ onto the corresponding parts of $\Omt$. Restricted to interior simplices with at most one node on the boundary, this map is the identity. On boundary simplices, $\Lambda_h$ is of class $C^{k}$ if the boundary is of class $C^{k}$ (see \cite[Lemma 4.6]{ER13}). \erev




\begin{definition}\label{def: lift}
\brev	For a function $\widehat{u}_h \in S_h(\bfx^*(t))$, we define its lift $\widehat{u}_h^\ell: \Omt \to \R$ by 
	\[ \widehat{u}_h^\ell(\Lambda_h(x,t),t) := \widehat{u}_h (x,t) \,. \]
	The composed lift from finite element functions $u_h$ on $\Omh(\bfx(t))$ to functions on $\Omt$ is denoted by
	\[ u_h^L = \widehat{u}_h^\ell \,. \] \erev
\end{definition}
For any $u \in H^{k+1}(\Om)$, there exists a unique finite element interpolation in the nodes $x_j^*$, denoted by $\widetilde{I}_h u \in S_h(\bfx^*)$. We set $I_h u=(\widetilde{I}_h u)^\ell: \Om \to \R$. An interpolation estimate is obtained from \cite[Proposition 5.4]{ER13}, based on \cite{Ber89}.
\begin{proposition}[Interpolation error]\label{lemma: interpolation}
	There exists a constant $c > 0$ independent of $h \le h_0$ ($h_0$ sufficiently small) and $t$ such that \brev for all $1 \le m \le k$, $u(\cdot,t) \in H^{m+1}(\Omt)$ and $t \in [0,T]$ 
	\[ \lVert u - I_h u \rVert_{L^2(\Omt)} + h \lVert \nb (u - I_h u) \rVert_{L^2(\Omt)} \le c h^{m} \lVert u \rVert_{H^{m+1}(\Omt)} \,. \]
\end{proposition}
\section{Statement of the main result}\label{section: main result}
We are now able to formulate the main result of this paper, which yields error bounds for the spatial semi-discretization of  \eqref{eq: weak formulation} with evolving isoparametric finite elements of polynomial degree $k \ge 2$.
We introduce the notation
\[ x_h^L(x,t) = X_h^L(p,t) \in \Omht \quad\text{for}\quad x=X(p,t)\in \Omt \,. \]

\begin{theorem}\label{theorem: error bound}
	Consider the spatial semi-discretization \eqref{eq: spatial semidiscretization} of \eqref{eq: weak formulation} with evolving isoparametric finite elements of order $k \ge 2$. We assume a quasi-uniform admissible triangulation of the initial domain and initial values chosen by finite element interpolations of the exact initial data. Assume that the problem admits an exact solution $(u,v,X)$ that is sufficiently smooth (\brev $u \in H^{k+1}(\Omt)$, $v, X \in H^{k+1}(\Omt)^n$, $n=2, 3$ \erev) for $t \in [0,T]$ \brev and a quasi-uniform triangulation of $\Om_0$. \erev Then there exists an $h_0 > 0$ such that for all mesh widths $h \le h_0$ the following error bounds hold on $\Omt$, for $t \in [0,T]$:\brev
	\begin{align}
		\left( \lVert u_h^L(\cdot,t) - u(\cdot,t) \rVert_{L^2(\Omt)}^2 + \int_0^t \lVert u_h^L(\cdot,s) - u(\cdot,s) \rVert_{H^1(\Om(s))}^2 \mathrm{d} s \right)^{\frac{1}{2}} & \le c h^k \,,\\
		\lVert v_h^L(\cdot,t) - v(\cdot,t) \rVert_{H^1(\Omt)^n} &\le c h^k \,, \\
		\lVert X_h^L(\cdot,t) - X(\cdot,t) \rVert_{H^1(\Om_0)^n} &\le c h^k \,.
	\end{align} \erev
	The constant $c$ depends on the regularity of the exact solution $(u,v,X)$, on $T$ and on the regularity of $f$.
\end{theorem}
\brev
In the following proof of error bounds, we clearly separate the stability and consistency analysis. The stability analysis, which is the significantly more difficult task in this work, borrows techniques used in \cite{KLLP17} and extends them to the present evolving bulk problem. The crucial differences are that in the stability analysis for the domain evolution the boundary has to be taken into account and the error only lives in the interior of the domain, whereas for the diffusion equation there is also an error on the boundary. The stability analysis relies on auxiliary results from Section \ref{section: auxiliary results}, which require a bound on the $W^{1,\infty}$-norm of the position errors. With the $H^1$-norm error bound together with an inverse estimate, we obtain an $\mathcal{O}(h^{k-n/2})$ error bound for the position error, which is only small for $k \ge 2$. This is why we impose the condition $k \ge 2$ in the above result. To apply this inverse estimate, we need that the interpolated exact domain $\Omh^*(t)$, which is the triangulation of $\Omt$ with the nodes $X(p_j,t)$, is quasi-uniform. Since the exact flow map $X(\cdot,t): \Om_0 \to \Omt$ is assumed to be smooth and non-degenerate, it is locally close to an invertible linear transformation, and therefore preserves admissibility of meshes on compact time intervals for sufficiently small $h \le h_0$, although the bounds in the admissibility inequalities and the largest possible mesh width may deteriorate with growing time. The boundedness of the $W^{1,\infty}$-norm of the position error is ensured with the $\mathcal{O}(h^k)$ error bound in $H^1$ norm that yields a $\mathcal{O}(h^{k-n/2})$ bound in the $W^{1,\infty}$ norm  with an inverse inequality. Therefore the assumptions of the theorem exclude a degeneration of the mesh for sufficiently small $h_0$.

The consistency analysis requires geometric estimates for the evolving isoparametric finite element method. Such estimates are mainly taken from \cite{ER13}, which are generalized to the time-dependent case in \cite{ER17pre}.
\erev

The stability proof will yield $h$-independent bounds of the errors in terms of the defects. The stability analysis is done in the matrix--vector formulation, which allows a compact and manageable representation of the computations. We use energy estimates and transport formulae to relate mass and stiffness matrices for different nodal vectors. This allows us to work with the interpolated exact domain $\Omh(\bfx^*(t))$, which is a finite element triangulation of $\Om(t)$ and only available in theoretical consideration.

In Section~\ref{section: auxiliary results} we prove auxiliary results that are used in the stability analysis, and then collect geometric estimates which are needed for the consistency analysis. In Section~\ref{section: stability velocity} we analyze stability of the semi-discrete velocity law without a diffusion equation on the evolving domain. The stability analysis of the semi-discrete diffusion equation, which requires results from Section~\ref{section: stability velocity}, is then done in Section~\ref{section: stability diffusion}. The defects are then bounded in Section~\ref{section: defect bounds} and the proof of Theorem \ref{theorem: error bound} is completed in Section~\ref{section: proof}.

\section{Auxiliary results}\label{section: auxiliary results}

\subsection{Properties of the evolving mass and stiffness matrix}
The following construction and results are similar to \cite[Section 4]{KLLP17}, \brev where similar identities are shown for surfaces only. We extend these results to the present case of domains. \erev In the stability analysis, we have to relate finite element matrices corresponding to different nodal vectors. Let $\bfx, \bfy \in \R^{3 N}$ be two nodal vectors defining discrete domains $\Omh(\bfx)$, $\Omh(\bfy)$, respectively. We set $\bfe = \bfx-\bfy$. For any $\theta \in [0,1]$, we have the intermediate domain $\Omh^\theta = \Omh(\bfy + \theta \bfe)$ which is the discrete domain corresponding to the intermediate nodal vector $\bfy + \theta \bfe$. 

For any vector $\bfw \in \R^N$, we set
\[ w_h^\theta = \sum_{j=1}^N w_j \varphi_j[\bfy + \theta \bfe] \in S_h(\bfy + \theta \bfe) \,.\]
In particular, we have the finite element function $e_h^\theta$ corresponding to $\bfe$:
\[ e_h^\theta = \sum_{j=1}^N e_j \varphi_j[\bfy + \theta \bfe] \,. \]
%
\begin{lemma}\label{lemma: matrix difference}
	In the above setting, the following identities hold for any $\bfw, \bfz \in \R^N$:
	\[ \bfw\T(\bfM(\bfx) - \bfM(\bfy))\bfz = \int_0^1 \int_{\Omh^\theta} w_h^\theta(\nb \cdot e_h^\theta) z_h^\theta \dx \dTheta \,,\]
	\[ \bfw\T(\bfA(\bfx) - \bfA(\bfy))\bfz = \int_0^1 \int_{\Omh^\theta} \nb w_h^\theta \cdot (D_{\Omh^\theta} e_h^\theta) \nb z_h^\theta \dx \dTheta \,,\]
	where $D_{\Omh^\theta} = \mathrm{trace}(\nb e_h^\theta) I_3 - (\nb e_h^\theta + (\nb e_h^\theta)\T)$.
\end{lemma}
\begin{proof}
	We use transport formulae from \cite[p. 23]{ER17pre}:
	\begin{align*}
		&\bfw\T (\bfA(\bfx) - \bfA(\bfy)) \bfz = \int_{\Omh(\bfx)} \nb w_h^1 \cdot \nb z_h^1 \dx - \int_{\Omh(\bfy)} \nb w_h^0 \cdot \nb z_h^0 \dx = \int_0^1 \dfdt{}{\theta} \int_{\Omh^\theta} \nb w_h^\theta \cdot \nb z_h^\theta \dx \dTheta \\
		&= \int_0^1 \int_{\Omh^\theta} \nb \mat_\theta w_h^\theta \cdot \nb z_h^\theta + \nb w_h^\theta \cdot \nb \mat_\theta z_h^\theta + \left( (\nb \cdot e_h^\theta) I_3 - \left(  \nb e_h^\theta + (\nb e_h^\theta)\T \right) \right)\nb w_h^\theta \cdot \nb z_h^\theta \dx \dTheta \,.
	\end{align*}
	The first two terms vanish thanks to the transport property. This shows the second identity, since $\nb \cdot e_h^\theta = \mathrm{trace}(\nb e_h^\theta)$. The first identity is proven similarly.
\end{proof}

A direct consequence is the following lemma, \brev where for any symmetric and positive (semi-)definite matrix $\bfK$, we denote the induced (semi-)norm on $\R^N$ by $\lVert \bfw \rVert_\bfK := (w\T \bfK \bfw)^{1/2}$.\erev
\begin{lemma}\label{lemma: Estimate matrix difference}
	If $\lVert \nb \cdot e_h^\theta \rVert_{L^\infty(\Omh^\theta)} \le \mu$ for $\theta \in [0,1]$, then
	\[ \lVert \bfw \rVert_{\bfM(\bfy+\theta \bfe)} \le e^{\frac{\mu}{2}} \lVert \bfw \rVert_{\bfM(\bfy)} \quad\text{for }\theta \in [0,1] \,. \]
	If $\lVert D_{\Omh^\theta} e_h^\theta \rVert_{L^\infty(\Omh^\theta)} \le \eta$ for $\theta \in [0,1]$, then
	\[ \lVert \bfw \rVert_{\bfA(\bfy + \theta \bfe)} \le e^{\frac{\eta}{2}} \lVert \bfw \rVert_{\bfA(\bfy)} \quad\text{for }\theta \in [0,1]\,.	 \]
\end{lemma}
\begin{proof}\brev
	We use the previous lemma and an $L^2$-$L^\infty$-$L^2$-estimate and compute for $0 \le \tau \le 1$:
	\begin{align}
	\lVert \bfw \rVert_{\bfM(\bfy + \tau \bfe)}^2 - \lVert \bfw \rVert_{\bfM(\bfy)}^2 &= \bfw\T \left( \bfM(\bfy+\tau\bfe) - \bfM(\bfy) \right) \bfw \\
	&= \int_0^\tau \int_{\Omh^\theta} w_h^\theta \nb \cdot e_h^\theta w_h^\theta \dx \dTheta \\
	&= \int_0^\tau \lVert \bfw \rVert_{\bfM(\bfy+\theta \bfe)}^2 \lVert \nb \cdot e_h^\theta \rVert _{L^\infty(\Omh^\theta)} \dTheta \\
	&\le \mu \int_0^\tau \lVert \bfw \rVert_{\bfM(\bfy+\theta \bfe)}^2 \dTheta \,.
	\end{align}
	A Gronwall argument shows the first result. The second estimate is shown analogously. \erev	
\end{proof}

\begin{lemma}\label{lemma: aux 2}
	Assume that
	\begin{align}\label{condition W1inf}
		\lVert \nb e_h^0 \rVert_{L^\infty(\Om_h(\bfy))} \le \frac{1}{2} \,.
	\end{align}
	Then, for $0 \le \theta \le 1$ the function $w_h^\theta = \sum_{j=1}^N w_j \varphi_j[\bfy + \theta \bfe]$ on $\Om_h^\theta$ is bounded by
	\[ \lVert \nb w_h^\theta \rVert_{L^p(\Om_h^\theta)} \le c_p \lVert \nb w_h^0 \rVert_{L^p(\Om_h^0)} \,, \]
	\[ \lVert w_h^\theta \rVert_{L^p(\Om_h^\theta)} \le \widetilde{c}_p \lVert w_h^0 \rVert_{L^p(\Om_h^0)} \,, \]
	for $1 \le p \le \infty$, where $c_p$ and $\widetilde{c}_p$ depend only on $p$.
\end{lemma}
\begin{proof}
	We parametrize $\Omh^\theta$ over $\Omh^0$:
	\begin{align*}
		Y_h^\theta(p_h) &= Y_h(p_h,\theta) = \sum_{j=1}^N (y_j + \theta e_j) \varphi_j[\bfy](p_h) \qquad \left(p_h \in \Omh^0 = \Omh(\bfy) \right) \\
		&=\sum_{j=1}^N y_j \varphi_j[\bfy](p_h) + \theta \sum_{j=1}^N e_j \varphi_j[\bfy](p_h) = p_h + \theta e_h^0(p_h) \,,
	\end{align*}
	where we have used that $Y_h^0(p_h) = p_h$. Differentiating with respect to $p_h$ yields
	\begin{align}\label{eq: misc 1}
		D Y_h^\theta(p_h) = I + \theta D e_h^0(p_h) \,.
	\end{align}
	By the transport property, we have $w_h^\theta(Y_h^\theta(p_h)) = w_h^0(Y_h^0(p_h)) = w_h^0(p_h)$. Differentiation with respect to $p_h$ yields
	\begin{align}\label{eq: misc 2}
		D w_h^\theta(Y_h^\theta(p_h)) D Y_h^\theta(p_h) = D w_h^0(p_h) \,.
	\end{align}
	From \eqref{eq: misc 1} we have under the assumption $\lVert \nb e_h^0 \rVert_{L^\infty(\Omh(\bfy))} \le \frac{1}{2}$:
	\begin{align*}
		| D Y_h^\theta(p_h) z| = |z + \theta (\nb e_h^0)\T z| \ge |z| - \theta |(\nb e_h^0)\T z| \ge \frac{1}{2} |z| \,.
	\end{align*}
	Thus, the matrix $D Y_h^\theta(p_h)$ is invertible and we have with \eqref{eq: misc 2}
	\[ D w_h^\theta(Y_h^\theta(p_h)) = D w_h^0(p_h) \left(D Y_h^\theta(p_h)\right)^{-1} \,,\]
	implying $|D w_h^\theta (Y_h^\theta(p_h))| \le 2 |D w_h^0(p_h)|$ and thus
	\[ \lVert \nb w_h^\theta \rVert_{L^\infty(\Omh^\theta)} \le 2 \lVert \nb w_h^0 \rVert_{L^\infty(\Omh^0)} \,.\]
	For $1 \le p < \infty$, we use the transformation formula and the fact that $\lVert D e_h^0(p_h) \rVert_{L^\infty(\Omh^0)} \le \frac{1}{2}$ to obtain
	\begin{align*}
		\lVert \nb w_h^\theta \rVert^p_{L^p(\Omh^\theta)} &= \int_{\Omh^\theta} | D w_h^\theta(y_h^\theta)|^p \mathrm{d} y_h^\theta =\int_{\Omh^0} \left| D w_h^\theta (Y_h^\theta(p_h)) \right|^p \left| \det D Y_h^\theta (p_h) \right| \mathrm{d} p_h \\
		&= \int_{\Omh^0}|D w_h^0(p_h) (D Y_h^\theta(p_h))^{-1}|^p |\det D Y_h(p_h)| \mathrm{d} p_h \\
		&\le c \int_{\Omh^0} |D w_h^0(p_h)|^p \mathrm{d} p_h = c \lVert \nb w_h^0 \rVert_{L^p(\Omh^0)}^p \,.
	\end{align*}
	For the second estimate, we note that the transport property immediately implies
	\[ \lVert w_h^\theta \rVert_{L^\infty(\Omh^\theta)} = \left\lVert w_h^0 \right\rVert_{L^\infty(\Omh^0)} \,. \]
	For $1 \le p < \infty$, we use the transformation formula and the same arguments as above.
\end{proof}

Another consequence of Lemma~\ref{lemma: Estimate matrix difference} is the following.
\begin{lemma}\label{lemma: Derivative estimate}
	Let $\bfx^*(t)$ be the vector of the exact positions $x_j^*(t) = X(x_j^0,t)$. Then, we have for all $\bfw,\bfz \in \R^N$:
	\[ \bfw\T \left( \dfdt{}{t} \bfM(\bfx^*(t)) \right) \bfz \le c \lVert \bfw \rVert_{\bfM(\bfx^*(t))} \lVert \bfz \rVert_{\bfM(\bfx^*(t))} \,,\]
	\[ \bfw\T \left( \dfdt{}{t} \bfA(\bfx^*(t)) \right) \bfz \le c \lVert \bfw \rVert_{\bfA(\bfx^*(t))} \lVert \bfz \rVert_{\bfA(\bfx^*(t))} \,. \]
	\brev The constant $c$ depends on the $W^{1,\infty}(\Om_T)$-norm of $v$ the dimension $n$ and the length $T$ of the time interval, but is independent of $h$ and $t$. \erev
\end{lemma}
\begin{proof}
	The proof can be found in \cite[Lemma 4.1]{DLM12} for surfaces and can directly be transferred to the present situation, using arguments from the proof of Lemma~\ref{lemma: Estimate matrix difference}.
\end{proof}

\subsection{Geometric estimates}
We collect geometric estimates that are used later in the consistency analysis. For a finite element function $\eta_h: \Omh^*(t) \to \R$, its lift is denoted by $\eta_h^\ell: \Om(t) \to \R$ (see Definition~\ref{def: lift}). The following lemma shows that the norms of finite element functions and their lifts are equivalent. A proof can be found in \cite[Proposition 4.9]{ER13}, based on \cite{CR72}.
\begin{lemma}\label{lem: norm equivalency}
	Let $\eta_h: \Omh^*(t) \to \R$ with lift $\eta_h^\ell: \Om(t) \to \R$. Then there exist constants $c_1, c_2 > 0$ such that
	\begin{alignat}{2}
		c_1 \lVert \eta_h \rVert_{L^2(\Omh^*(t))} & \le \lVert \eta_h^\ell \rVert_{L^2(\Om(t))} &&\le c_2 \lVert \eta_h \rVert_{L^2(\Omh^*(t))} \,, \\
		c_1 \lVert \nb \eta_h \rVert_{L^2(\Omh^*(t))} & \le \lVert \nb \eta_h^\ell \rVert_{L^2(\Om(t))} &&\le c_2 \lVert \nb \eta_h \rVert_{L^2(\Omh^*(t))} \,.
	\end{alignat}
	\brev The constant $c$ depends on the dimension $n$, the length $T$ of the time interval and the geometry of $\Om_T$ but is independent of $h$ and $t$. \erev
\end{lemma}

We define discrete analogues of the bilinear forms $m$ and $a$, defined in \eqref{eq: bilinear forms}: For $\eta_h$, $\chi_h$: $\Omh^*(t) \to \R$, we define
\begin{align}
	m_h^*(\eta_h,\chi_h) &= \int_{\Omh^*(t)} \eta_h \chi_h \,,\\
	a_h^*(\eta_h,\chi_h) &= \int_{\Omh^*(t)} \nb \eta_h \cdot \nb \chi_h \,.
\end{align}

\brev The following lemma estimates the difference between the discrete bilinear form on the interpolated exact domain and the exact bilinear form of the lifted functions on the exact domain. A proof can be found in \cite{ER17pre}. \erev

\begin{lemma}[Geometric approximation errors]\label{lemma: error bilinear forms} For $\eta_h, \chi_h \in S_h(\bfx^*(t))$ with corresponding lifts $\eta_h^\ell$, $\chi_h^\ell$, the following estimates hold: there exists a constant $c$ such that
	\begin{align}
		\left| m_h^*(\eta_h,\chi_h) - m(\eta_h^\ell,\chi_h^\ell) \right| &\le c h^k \lVert \eta_h^\ell \rVert_{L^2(\Omt)} \lVert \chi_h^\ell  \rVert_{L^2(\Omt)} \,, \\
		\left| a_h^*(\eta_h,\chi_h) - a(\eta_h^\ell,\chi_h^\ell) \right| &\le c h^k \lVert \nb \eta_h^\ell \rVert_{L^2(\Omt)} \lVert \nb \chi_h^\ell \rVert_{L^2(\Omt)} \,.
	\end{align}	
	\brev The constant $c$ depends on the dimension $n$, the length $T$ of the interval and the geometry of $\Om_T$ but is independent of $h$ and $t$. \erev
\end{lemma}

\section{Stability of the semi-discrete harmonic velocity law}\label{section: stability velocity}
We will start with analyzing stability of the semi-discrete velocity law without the diffusion equation, since the domain evolution is \brev independent of the parabolic equation, see Remark~\ref{remark: domain motion independent}\erev. The stability analysis of the semi-discrete diffusion equation, which is based on the following results, is presented in the next section.

We consider the nodal vectors $\bfv, \bfx \in \R^{3N}$ which satisfy
\begin{align}\label{eq: ODE system}
	\begin{aligned}
		(I_3 \otimes \bfA_{22}(\bfx)) \bfv^\Om &= -(I_3 \otimes \bfA_{21}(\bfx)) \bfv^\Ga \,, \\
		\dot{\bfx} = \bfv.
	\end{aligned}
\end{align}
with given $\bfv^\Ga$. We denote by
\[ \bfx^*(t) =\begin{pmatrix} \bfx^{\Ga,*}(t) \\ \bfx^{\Om,*}(t) \end{pmatrix} \]
the vector of the exact positions at time $t \in [0,T]$. Note that $x_j^*(t) = x_j(t)$ for all $j = 1,\ldots,N_\Ga$ since $\bfv^\Ga$ is given explicitly. i.e. $\bfx^\Ga(t) = \bfx^{\Ga,*}(t)$. 

We consider the interpolated exact velocity $v_h^*(\cdot,t) = \sum_{j=1}^N v_j^*(t) \varphi_j[\bfx^*(t)]$ with the corresponding nodal vector
\[ \bfv^*(t) = \begin{pmatrix} \bfv^{\Ga,*}(t) \\ \bfv^{\Om,*}(t) \end{pmatrix}\,.\]
Note again that $\bfv^{\Ga,*}(t) = \bfv^{\Ga}(t)$.

\subsection{Error equations}
The vectors $\bfx^*$ and $\bfv^*$ satisfy \eqref{eq: ODE system} up to a defect $\bfd_{\bfv^\Om}$:
\begin{align}\label{eq: ODE system defects}
\begin{aligned}
(I_3 \otimes \bfA_{22}(\bfx^*)) \bfv^{\Om,*} &= - (I_3 \otimes \bfA_{21}(\bfx^*)) \bfv^{\Ga,*} + \bfM_{22}(\bfx^*) \bfd_{\bfv^\Om} \,, \\
\dot{\bfx}^* &= \bfv^* \,.
\end{aligned}
\end{align}
We set $\bfd_\bfv = (\bfd_{\bfv^\Ga},\bfd_{\bfv^\Om}) \in \R^{3N}$ with $\bfd_{\bfv^\Ga} = 0 \in \R^{3 N_\Ga}$. This notation will be useful in the stability analysis. The defect $\bfd_{\bfv}$ corresponds to a finite element function $d_h^v(\cdot,t) = \sum_{j=1}^N d_j^v(t) \varphi_j[\bfx^*(t)] \in \Sohxt^3$. We denote the errors in the nodes and in the velocity by $\bfe_{\bfx^\Om} = \bfx^\Om - \bfx^{\Om,*}$, $\bfe_{\bfv^\Om} = \bfv^\Om - \bfv^{\Om,*}$ and use the notation
\[ \bfe_\bfx = \begin{pmatrix} 0 \\ \bfe_{\bfx^\Om} \end{pmatrix} = \begin{pmatrix} \bfe_{\bfx^\Ga} \\ \bfe_{\bfx^\Om}\end{pmatrix} \,,\quad  
\bfe_\bfv = \begin{pmatrix} 0 \\ \bfe_{\bfv^\Om} \end{pmatrix} = \begin{pmatrix} \bfe_{\bfv^\Ga} \\ \bfe_{\bfv^\Om}\end{pmatrix} \,.\]
In the following, we write $\bfA(\bfx)$ instead of $I_3 \otimes \bfA(\bfx)$, for brevity. We rewrite \eqref{eq: ODE system} as
\begin{align}\label{eq: ODE system rewrite}
	\bfA_{22}(\bfx^*) \bfv^\Om = -(\bfA_{22}(\bfx) - \bfA_{22}(\bfx^*)) \bfv^{\Om,*}
	- (\bfA_{22}(\bfx) - \bfA_{22}(\bfx^*))\bfe_{\bfv^\Om} - \bfA_{21}(\bfx) \bfv^\Ga \,.
\end{align}
Subtracting \eqref{eq: ODE system defects} from \eqref{eq: ODE system rewrite} and using $\bfv^\Ga = \bfv^{\Ga,*}$ yields the error equations
\begin{equation}\label{eq: error}
\begin{aligned}
\bfA_{22}(\bfx^*) \bfe_{\bfv^\Om} & = &&-(\bfA_{22}(\bfx) - \bfA_{22}(\bfx^*)) \bfv^{\Om,*} - (\bfA_{22}(\bfx) - \bfA_{22}(\bfx^*)) \bfe_{\bfv^\Om} \\
& &&-(\bfA_{21}(\bfx) - \bfA_{21}(\bfx^*)) \bfv^{\Ga,*} - \bfM_{22}(\bfx^*) \bfd_{\bfv^\Om} \,, \\
\dot{\bfe}_{\bfx^\Om} &=&&\bfe_{\bfv^\Om} \,.
\end{aligned}
\end{equation}

\subsection{Dual norms}
We recall that the mass and stiffness matrices $\bfM_{22}(\bfx)$ and $\bfA_{22}(\bfx)$, respectively, induce norms on $\Sohx$. Note that $\bfA_{22}(\bfx)$ defines a norm on $\Sohx$, whereas $\bfA(\bfx)$ defines only a semi-norm on $S_h(\bfx)$.
We define the dual norm
\begin{align}\label{eq: dual norm}
	\begin{aligned}
		\lVert d_h^v \rVert_{H_{0,h}^{-1}(\Omh(\bfx^*))} &= \sup_{0 \ne \psi_h \in S_{0,h}(\bfx^*)^3} \frac{\int_{\Omh(\bfx^*)} d_h^v \cdot \psi_h \dx}{\lVert \psi_h \rVert_{H_0^1(\Omh(\bfx^*))}} = \sup_{0 \ne \bfz \in \R^{3 N_\Om}} \frac{\bfd_{\bfv^\Om}\T \bfM_{22}(\bfx^*) \bfz}{(\bfz\T \bfA_{22}(\bfx^*) \bfz)^{1/2}} \\
		&=\sup_{0 \ne \bfw \in \R^{3 N_\Om}} \frac{\bfd_{\bfv^\Om}\T \bfM_{22}(\bfx^*) \bfA_{22}(\bfx^*)^{-1/2} \bfw}{(\bfw\T \bfw)^{1/2}}  =\lVert \bfA_{22}(\bfx^*)^{-1/2} \bfM_{22}(\bfx^*) \bfd_{\bfv^\Om} \rVert_2 \\
		&=\left( \bfd_{\bfv^\Om}\T \bfM_{22}(\bfx^*) \bfA_{22}(\bfx^*)^{-1} \bfM_{22}(\bfx^*) \bfd_{\bfv^\Om}\right)^{\frac{1}{2}} =: \lVert \bfd_{\bfv^\Om} \rVert_{\star,\bfx^*} \,.
	\end{aligned}
\end{align}

\subsection{Stability estimate}
We are now ready to state and prove the first main stability result. The following stability result holds under a smallness assumption on the defect. It will be proven in Section~\ref{section: defect bounds} that this assumption is satisfied for $\kappa = k \ge 2$, where $k$ is the order of the finite element method.

\begin{lemma}\label{lem: stability}
	Assume that, for some $\kappa > \frac{3}{2}$, the defect is bounded as follows:
	\begin{align}\label{eq: defect assumption1}
		\lVert \bfd_{\bfv^\Om}(t) \rVert_{\star,\bfx^*(t)} \le c h^\kappa \,,\quad t \in [0,T] \,.
	\end{align}
	Then there exists an $h_0 > 0$ such that for $h \le h_0$ and $t \in [0,T]$, the following error bounds hold.
	\begin{align}
		\lVert \bfe_{\bfx^\Om}(t) \rVert_{\bfA_{22}(\bfx^*(t))}^2 &\le c \int_0^t \lVert \bfd_{\bfv^\Om}(s) \rVert_{\star,\bfx^*(s)}^2 \mathrm{d} s \,, \label{error bound 1} \\
		\lVert \bfe_{\bfv^\Om}(t) \rVert_{\bfA_{22}(\bfx^*(t))}^2 &\le c \lVert \bfd_{\bfv^\Om}(t) \rVert_{\star,\bfx^*(t)}^2 + c \int_0^t \lVert \bfd_{\bfv^\Om}(s) \rVert_{\star,\bfx^*(s)}^2 \mathrm{d} s \label{error bound 2}\,.
	\end{align}
\end{lemma}
\begin{proof}
	The proof uses \brev energy estimates that are similar to techniques used in \cite{KLLP17} and \cite{KLL18MCF}. We extend their results for coupled surface problems to the present evolving bulk problem. Since the structure of the proof is similar to the cited works, we might skip some non-trivial steps. However there are some crucial differences that need to be pointed out: The evolving bulk $\Omt$ has a boundary $\Gat$ that has to be taken into account, whereas in \cite{KLLP17,KLL18MCF} the considered evolving surfaces have no boundaries or interiors. Moreover, we exploit the fact there is no position or velocity error in the boundary because the boundary velocity is given. This implies that the lift of the finite element function corresponding to the error is a $H_0^1$-function and turns out to be crucial to estimate the error equations, see \eqref{eq: aux combine} and \eqref{eq: milestone} below. In addition, the space dimension $n \in \{2,3\}$ requires the assumption $\kappa > n/2$, which is due to an inverse estimate at the end of this proof, see Remark~\ref{remark: dimension 2}. 
	
	In view of the auxiliary results from Section \ref{section: auxiliary results} and in particular condition \eqref{condition W1inf}, we need to control the $W^{1,\infty}$-norm of the position error $e_x(\cdot,t)$. Let \erev $0 < t^* \le T$ be the maximal time such that
	\begin{align}\label{ineq: t-stern}
	\lVert \nb e_x(\cdot,t) \rVert_{L^\infty(\Omh(\bfx^*(t)))} \le h^{(\kappa-3/2)/2} \,.
	\end{align}
	Note that $e_x(\cdot,0)=0$ implies $t^* > 0$. We prove the stated error bounds for $t \in [0,t^*]$ and then show that $t^* = T$.
	
	We test the first equation of \eqref{eq: error} with $\bfe_{\bfv^\Om}$ and obtain
	\begin{align}\label{eq: tested error equation}
		\begin{aligned}
			\lVert \bfe_{\bfv^\Om} \rVert_{\bfA_{22}(\bfx^*)}^2 = &-\bfe_{\bfv^\Om}\T ( \bfA_{22}(\bfx) - \bfA_{22}(\bfx^*)) \bfv^{\Om,*} - \bfe_{\bfv^\Om}\T ( \bfA_{22}(\bfx) - \bfA_{22}(\bfx^*)) \bfe_{\bfv^\Om} \\
			&- \bfe_{\bfv^\Om}\T (\bfA_{21}(\bfx) - \bfA_{21}(\bfx^*)) \bfv^{\Ga,*} -\bfe_{\bfv^\Om}\T \bfM_{22}(\bfx^*) \bfd_{\bfv^\Om} \,.
		\end{aligned}
	\end{align}

	It is crucial to combine the first and third term of \eqref{eq: tested error equation}. Note that
	\begin{align}\label{eq: aux combine}
	\begin{aligned}
		(0,\bfe_{\bfv^\Om}\T) \bfA(\bfx) \begin{pmatrix} \bfv^\Ga \\ \bfv^\Om \end{pmatrix} &= (0,\bfe_{\bfv^\Om}\T) \begin{pmatrix} \bfA_{11}(\bfx) & \bfA_{12}(\bfx) \\ \bfA_{21}(\bfx) & \bfA_{22}(\bfx) \end{pmatrix} \begin{pmatrix} \bfv^\Ga \\ \bfv^\Om \end{pmatrix} \\
		&= \bfe_{\bfv^\Om}\T \bfA_{21}(\bfx) \bfv^\Ga+\bfe_{\bfv^\Om}\T \bfA_{22}(\bfx) \bfv^\Om \,.
	\end{aligned}
	\end{align}
	We set $\bfe_{\bfv^\Ga} = 0$ and $\bfe_\bfv\T = (\bfe_{\bfv^\Ga}\T,\bfe_{\bfv^\Om}\T)$. Applying \eqref{eq: aux combine} to \eqref{eq: tested error equation}, we obtain
	\begin{align}\label{eq: milestone}
	\lVert \bfe_\bfv \rVert_{\bfA(\bfx^*)}^2 = &- \bfe_\bfv\T(\bfA(\bfx)-\bfA(\bfx^*)) \bfv^* - \bfe_{\bfv^\Om}\T (\bfA_{22}(\bfx) - \bfA_{22}(\bfx^*)) \bfe_{\bfv^\Om} - \bfe_{\bfv^\Om}\T \bfM_{22}(\bfx^*) \bfd_{\bfv^\Om} \\
	=&-\bfe_\bfv\T(\bfA(\bfx)-\bfA(\bfx^*)) \bfv^* -\bfe_{\bfv}\T(\bfA(\bfx) - \bfA(\bfx^*)) \bfe_{\bfv} -\bfe_\bfv\T \bfM(\bfx^*) \bfd_\bfv \,.
	\end{align}
	We estimate these three terms separately.
	
	\brev (i) We use that
	\[ D_{\Omh^\theta} e_x^\theta = \mathrm{trace}(\nb e_x^\theta) I_3 -\left( \nb e_x^\theta + (\nb e_x^\theta)\T \right) \]
	and thus $\lVert D_{\Omh^\theta} e_x^\theta \rVert \le c \lVert \nb e_x^\theta \rVert$. With Lemma~\ref{lemma: matrix difference}, an $L^2$-$L^2$-$L^\infty$-estimate and Lemma~\ref{lemma: Estimate matrix difference}, we arrive at
	\begin{align}
		\bfe_\bfv\T \left( \bfA(\bfx) - \bfA(\bfx^*) \right) \bfv^* &= \int_0^1 \int_{\Omh^\theta} \nb e_v^\theta \cdot ( D_{\Omh^\theta} e_x^\theta) \nb v_h^{*,\theta} \dx \dTheta \\
		&\le \int_0^1 \lVert \nb e_v^\theta \rVert_{L^2(\Omh^\theta)} \lVert D_{\Omh^\theta} e_x^\theta \rVert_{L^2(\Omh^\theta)} \lVert \nb v_h^{*,\theta} \rVert_{L^\infty(\Omh^\theta)} \dTheta \\
		&\le c \lVert \nb e_v^0 \rVert_{L^2(\Omh^0)} \lVert \nb e_x^0 \rVert_{L^2(\Omh^0)} \lVert \nb v_h^{*,0} \lVert_{L^\infty(\Omh^0)} \\
		&=c \lVert \bfe_\bfv \rVert_{\bfA(\bfx^*)} \lVert \bfe_\bfx \rVert_{\bfA(\bfx^*)} \lVert \nb v_h^* \rVert_{L^\infty(\Omh(\bfx^*))} \,.
	\end{align}
	\erev
	The last factor is bounded by a constant independent of $h$, since $v_h^*$ is the finite element interpolation of the exact velocity (see \cite[Theorem 4.1]{Ber89}). Using Young's inequality together with the fact that $\lVert \bfe_\bfv \rVert_{\bfA(\bfx^*)} = \lVert \bfe_{\bfv^\Om} \rVert_{\bfA_{22}(\bfx^*)}$, we obtain
	\[ \bfe_{\bfv}\T (\bfA(\bfx) - \bfA(\bfx^*)) \bfv^* \le \frac{1}{4} \lVert \bfe_{\bfv^\Om} \rVert_{\bfA_{22}(\bfx^*)}^2 + C \lVert \bfe_{\bfx^\Om} \rVert_{\bfA_{22}(\bfx^*)}^2 \,.\]
	
	(ii) Similarly, using the smallness assumption \eqref{ineq: t-stern}, we obtain
	\begin{align*}
		\bfe_{\bfv}\T(\bfA(\bfx) - \bfA(\bfx^*)) \bfe_{\bfv} &\le c \lVert \nb e_v^0 \rVert_{L^2(\Omh(\bfx^*))}^2 \lVert \nb e_x^0 \rVert_{L^\infty(\Omh(\bfx^*))} \\
		&\le c h^{(\kappa-3/2)/2} \lVert \bfe_\bfv \rVert_{\bfA(\bfx^*)}^2 = c h^{(\kappa-3/2)/2} \lVert \bfe_{\bfv^\Om} \rVert_{\bfA_{22}(\bfx^*)}^2 \,.
	\end{align*}
	
	(iii) Using the Cauchy--Schwarz inequality together with Young's inequality, we estimate
	\begin{align*}
		\bfe_{\bfv^\Om}\T \bfM_{22}(\bfx^*) \bfd_{\bfv^\Om} &= \bfe_{\bfv^\Om}\T \bfA_{22}(\bfx^*)^{\frac{1}{2}} \bfA_{22}(\bfx^*)^{-\frac{1}{2}} \bfM_{22}(\bfx^*) \bfd_{\bfv^\Om} \\
		&\le \frac{1}{4} \lVert \bfA_{22}(\bfx^*)^{\frac{1}{2}} \bfe_{\bfv^\Om} \rVert^2 + c\lVert \bfA_{22}(\bfx^*)^{-\frac{1}{2}} \bfM_{22}(\bfx^*) \bfd_{\bfv^\Om} \rVert^2 \\
		&= \frac{1}{4} \lVert \bfe_{\bfv^\Om} \rVert_{\bfA_{22}(\bfx^*)}^2 + c \lVert \bfd_{\bfv^\Om} \rVert_{\star,\bfx^*}^2 \,.
	\end{align*}
	The combination of the three estimates with absorptions (for $h \le h_0$ sufficiently small) yields
	\begin{align}\label{ineq: milestone 1}
		\lVert \dot{\bfe}_{\bfx^\Om} \rVert_{\bfA_{22}(\bfx^*)}^2 = \lVert \bfe_{\bfv^\Om} \rVert_{\bfA_{22}(\bfx^*)}^2 \le c \lVert \bfe_{\bfx^\Om} \rVert_{\bfA_{22}(\bfx^*)}^2 + c \lVert \bfd_{\bfv^\Om} \rVert_{\star,\bfx^*}^2 \,.
	\end{align}
	We connect $\dfdt{}{t} \lVert \bfe_{\bfx^\Om} \rVert_{\bfA_{22}(\bfx^*)}^2$ and $\lVert \dot{\bfe}_{\bfx^\Om} \rVert_{\bfA_{22}(\bfx^*)}^2$. \brev We have
	\[ \frac{1}{2} \dfdt{}{t} \lVert \bfe_{\bfx^\Om} \rVert_{\bfA_{22}(\bfx^*)}^2 = \bfe_{\bfx^\Om}\T \bfA_{22}(\bfx^*) \dot{\bfe}_{\bfx^\Om} + \frac{1}{2} \bfe_{\bfx^\Om}\T \left( \dfdt{}{t} \bfA_{22}(\bfx^*(t)) \right) \bfe_{\bfx^\Om} \,.\]
	With the Cauchy--Schwarz and Young inequalities, we obtain
	\[ \bfe_{\bfx^\Om}\T \bfA_{22}(\bfx^*) \dot{\bfe}_{\bfx^\Om} \le \lVert \bfe_{\bfx^\Om} \rVert_{\bfA_{22}(\bfx^*)} \lVert \dot{\bfe}_{\bfx^\Om} \rVert_{\bfA_{22}(\bfx^*)} \le \lVert \dot{\bfe}_{\bfx^\Om} \rVert_{\bfA_{22}(\bfx^*)}^2 + \frac{1}{4} \lVert \bfe_{\bfx^\Om} \rVert_{\bfA_{22}(\bfx^*)}^2\,.\]
	For the second term, Lemma \ref{lemma: Derivative estimate} yields
	\[ 
	\frac{1}{2} \bfe_{\bfx^\Om}\T \left( \dfdt{}{t} \bfA_{22}(\bfx^*(t)) \right) \bfe_{\bfx^\Om} \le C \lVert \bfe_{\bfx^\Om}(t) \rVert_{\bfA_{22}(\bfx^*(t))}^2 \,.\]
	We thus obtain, using \eqref{ineq: milestone 1}
	\begin{align}\label{eq: aux 13}
		\frac{1}{2} \dfdt{}{t} \lVert \bfe_{\bfx^\Om} \rVert_{\bfA_{22}(\bfx^*)}^2 \le c \lVert \bfe_{\bfx^\Om} \rVert_{\bfA_{22}(\bfx^*)}^2 + c \lVert \bfd_{\bfv^\Om} \rVert_{\star,\bfx^*}^2 \,.
	\end{align}
	Integrating from $0$ to $t$ and using $\bfe_{\bfx^\Om}(0)=0$, we obtain
	\[ \lVert \bfe_{\bfx^\Om}(t) \rVert_{\bfA_{22}(\bfx^*(t))}^2 \le c \int_0^t \lVert \bfd_{\bfv^\Om}(s) \rVert_{\star,\bfx^*(s)}^2 \mathrm{d}s + \int_0^t c \lVert \bfe_{\bfx^\Om}(s) \rVert_{\bfA_{22}(\bfx^*(s))}^2 \mathrm{d} s \,.\]
	The Gronwall inequality thus yields \eqref{error bound 1}, which then inserted into \eqref{ineq: milestone 1} yields \eqref{error bound 2}.\erev
	
	Now it remains to show that for $h \le h_0$ sufficiently small we in fact have $t^* = T$. For $0 \le t \le t^*$, we have with an inverse inequality (see \cite{BrennerScott}) and for $h \le h_0$ sufficiently small:
	\begin{align*}
	\lVert \nb e_x(\cdot,t) \rVert_{L^\infty(\Omh(\bfx^*(t)))} &\le c h^{-3/2} \lVert \nb e_x(\cdot,t) \rVert_{L^2(\Omh(\bfx^*(t)))} \\
	&= c h^{-3/2} \lVert \bfe_{\bfx^\Om}(t) \rVert_{\bfA_{22}(\bfx^*(t))}^2 \le c h^{\kappa -3/2} \\
	&\le \frac{1}{2} h^{\frac{\kappa-3/2}{2}} \,.
	\end{align*}
	This shows that the bound \eqref{ineq: t-stern} can be extended beyond $t^*$, which contradicts the maximality of $t^*$ unless $t^* = T$.
\end{proof}

\begin{remark}\label{remark: dimension 2}
	The previous lemma remains valid in the two-dimensional case, where the assumption \eqref{eq: defect assumption1} is only required for $\kappa>1$. 
	Either way, it requires the finite element method to be of order two, at least.
\end{remark}

\section{Stability of the semi-discrete diffusion equation}\label{section: stability diffusion}
In this section, we extend the stability result to the nodal vector $\bfu(t)$ of the numerical solution to the semi-discrete diffusion equation.

\subsection{Error equations}
The numerical solution $u_h(x,t) = \sum_{j=1}^N u_j(t) \varphi_j[\bfx(t)](x)$ with corresponding nodal vector $\bfu=\bfu(t) = (u_j(t))_{j=1}^N$ satisfies 
\begin{align}\label{eq: matrix vector u}
	\dfdt{}{t} \left( \bfM(\bfx) \bfu \right) + \bfA(\bfx) \bfu = \bff(\bfx) \,.
\end{align}
The finite element interpolation $u_h^*(\cdot,t)$ of the exact solution $u(\cdot,t)$ with corresponding nodal vector $\bfu^*(t)$, when inserted into the matrix--vector formulation, yields defects $\bfd_\bfu$, corresponding to a finite element function $d_h^u$, such that
\begin{align}\label{eq: defect du}
	\dfdt{}{t} \left( \bfM(\bfx^*) \bfu^* \right) + \bfA(\bfx^*) \bfu^* = \bff(\bfx^*) + \bfM(\bfx^*) \bfd_\bfu \,.
\end{align}
Rewriting \eqref{eq: matrix vector u} in a similar way as \eqref{eq: ODE system rewrite} and subtracting from \eqref{eq: defect du} yields the error equation
\begin{align}\label{eq: error equation u}
	\begin{aligned}
		\dfdt{}{t} \big( \bfM(\bfx^*) \bfe_\bfu \big) + \bfA(\bfx^*) \bfe_\bfu = &-\dfdt{}{t} \big( (\bfM(\bfx)-\bfM(\bfx^*)) \bfu^* \big) - \dfdt{}{t} \big( (\bfM(\bfx)-\bfM(\bfx^*)) \bfe_\bfu \big) \\
		&-\big( \bfA(\bfx)-\bfA(\bfx^*) \big) \bfu^* - \big( \bfA(\bfx) - \bfA(\bfx^*) \big) \bfe_\bfu +(\bff(\bfx)-\bff(\bfx^*)) - \bfM(\bfx^*) \bfd_\bfu \,.
	\end{aligned}
\end{align}

\subsection{Dual norm}
In order to bound the defect in $u$, we need to introduce a different dual norm than in the previous section, which is due to the fact that the defect $d_h^u$ lives on the whole domain $\Omt$ and does not vanish on the boundary. We use the notation $\bfK(\bfx^*)=\bfM(\bfx^*)+\bfA(\bfx^*)$ and consider the dual norm (cf. \eqref{eq: dual norm})
\begin{align}\label{eq: dual norm u}
	\begin{aligned}
		\lVert d_h^u \rVert_{H_h^{-1} (\Omh(\bfx^*))} &= \sup_{0 \ne \psi_h \in S_h(\bfx^*)} \frac{ \int_{\Omh(\bfx^*)} d_h^u \psi_h \dx }{ \lVert \psi_h \rVert_{H^1(\Omh(\bfx^*))} } \\
		& = \left( \bfd_\bfu\T \bfM(\bfx^*) \bfK(\bfx^*)^{-1} \bfM(\bfx^*) \bfd_\bfu \right)^{1/2} =: \lVert \bfd_\bfu\rVert_{\star,\bfx^*} \,.
	\end{aligned}
\end{align}
For simplicity, we do not use another notation for the dual norm of $\bfd_\bfu$, as it will always be clear from context which dual norm is meant. In the following stability proof, we need the following technical lemma.
\begin{lemma}\label{lemma: mat grad}
	For a function $w=w(x,t): \Omt \to \R^3$, we have 
	\begin{align}
	\mat \left(\nb \cdot w\right) = \nb \cdot \mat w - \nb v \cdot \nb w \,,
	\end{align}
	where $v = v(x,t)$ is the velocity and $\nb v \cdot \nb w$ denotes the Frobenius norm inner product, i.e.~the Euclidean product of the vectorizations of the matrices.
\end{lemma}
\begin{proof}
	Based on \cite{DKM13}, a similar identity for the surface divergence is shown in \cite{KLLP17}. The proof is adapted by embedding everything into a surface $\Gat = \Omt \times \{0\} \in \R^4$.
%
\end{proof}

\subsection{Stability estimate}
We are now able to state and prove the stability result for the error $\bfe_\bfu$. Note that the previous stability estimates for $\bfe_\bfx$ and $\bfe_\bfv$ remain valid since the solution to the domain evolution does not depend on the numerical solution $u_h$, but the solution $u_h$ to the diffusion equation depends on the solution $\bfx$ of the position vectors, which is reflected in the following proof.
\begin{lemma}\label{lemma: stability u}
	Assume that, for some $\kappa > \frac{3}{2}$, the defects are bounded as follows:
	\begin{align}\label{eq: defect assumption2}
		\lVert \bfd_{\bfu} (t) \rVert_{\star,\bfx^*(t)} \le c h^\kappa \,,\quad  \lVert \bfd_{\bfv^\Om}(t) \rVert_{\star,\bfx^*(t)} \le c h^\kappa \,,\quad t \in [0,T] \,.
	\end{align}
	Then there exists an $h_0 > 0$ such that the following estimate holds for $h \le h_0$ and $t \in [0,T]$, where the constant $C$ is independent of $h$:
	\begin{align}\label{ineq: stability estimate 2}
	\lVert \bfe_\bfu(t) \rVert_{\bfM(\bfx^*)}^2 + \int_0^t \lVert \bfe_\bfu(s) \rVert_{\bfA(\bfx^*(s))}^2 \mathrm{d}s \le C \int_0^t \lVert \bfd_\bfu(s) \rVert_{\star,\bfx^*(s)}^2 + \lVert \bfd_\bfv(s) \rVert_{\star,\bfx^*(s)}^2 \mathrm{d} s\,.
	\end{align}
\end{lemma}
\begin{proof}
	The proof is similar to the proof of Lemma~\ref{lem: stability}. Let $0 < t^* \le T$ be the maximal time such that
	\begin{align}\label{ineq: t-stern coupled}
	\lVert \nb e_x(\cdot,t) \rVert_{L^\infty(\Omh(\bfx^*(t)))} &\le h^{(\kappa-3/2)/2} \,,\\
	\lVert e_u(\cdot,t) \rVert_{L^\infty(\Omh(\bfx^*(t)))} &\le 1 \,.
	\end{align}
	for all $t \in [0,t^*]$. Note that $e_x(\cdot,0)=0=e_u(\cdot,0)$ implies $t^*>0$. Again, we will prove the error bound for $t \in [0,t^*]$ and then show that $t^*$ coincides with $T$.
	
	Testing \eqref{eq: error equation u} with $\bfe_\bfu\T$, we obtain (omitting the argument $t$)
	\begin{align}\label{eq: aux 3}
	\begin{aligned}
		\bfe_\bfu\T \dfdt{}{t} \left( \bfM(\bfx^*) \bfe_\bfu \right) + \bfe_\bfu\T \bfA(\bfx^*) \bfe_\bfu = \, &- \bfe_\bfu\T \dfdt{}{t} \left( (\bfM(\bfx)-\bfM(\bfx^*)) \bfu^* \right)  -\bfe_\bfu\T \dfdt{}{t} \left( (\bfM(\bfx)-\bfM(\bfx^*)) \bfe_\bfu \right) \\
		&-\bfe_\bfu\T \left( \bfA(\bfx)-\bfA(\bfx^*) \right) \bfu^* -\bfe_\bfu\T \left( \bfA(\bfx)-\bfA(\bfx^*) \right) \bfe_\bfu \\
		&-\bfe_\bfu\T(\bff(\bfx)-\bff(\bfx^*)) -\bfe_\bfu\T \bfM(\bfx^*) \bfd_\bfu \,.
	\end{aligned}
	\end{align}
	\brev We estimate the six terms on the right-hand side separately.
	
	(i) We apply the product rule to obtain
	\begin{align}\label{eq: aux 5}
	\bfe_\bfu\T \dfdt{}{t} \left( ( \bfM(\bfx)-\bfM(\bfx^*))\bfu^* \right) &= \bfe_\bfu\T \left( \bfM(\bfx)-\bfM(\bfx^*)\right) \dot{\bfu}^* + \bfe_\bfu\T \left( \dfdt{}{t} \left( \bfM(\bfx) - \bfM(\bfx^*) \right) \right) \bfu^* \,.
	\end{align}
	For the first term of \eqref{eq: aux 5}, we use Lemma~\ref{lemma: matrix difference}, an $L^2$-$L^2$-$L^\infty$-estimate and Lemma~\ref{lemma: aux 2} to obtain
	\begin{align}
	\begin{aligned}
	\left| \bfe_\bfu\T \left( \bfM(\bfx) - \bfM(\bfx^*) \right) \dot{\bfu}^* \right| &= \left| \int_0^1 \int_{\Omh^\theta} e_u^\theta (\nb \cdot e_x^\theta) \mat_h u_h^{*,\theta} \dx \dTheta \right| \\
	&\le \int_0^1 \left\lVert e_u^\theta \right\rVert_{L^2(\Omh^\theta)} \left\lVert \nb \cdot e_x^\theta \right\rVert_{L^2(\Omh^\theta)} \left\lVert \mat_h u_h^{*,\theta} \right\rVert_{L^\infty(\Omh^\theta)} \dTheta \\
	&\le c \left\lVert \bfe_\bfu\right\rVert_{\bfM(\bfx^*)} \left\lVert \bfe_\bfx \right\rVert_{\bfA(\bfx^*)} \left\lVert \mat_h u_h^{*,0} \right\rVert_{L^\infty(\Omh(\bfx^*))} \,.
	\end{aligned}
	\end{align}
	With an elementary computation, the last term can be bounded by $\lVert \mat_h u_h^{*,0} \rVert_{L^\infty(\Omh(\bfx^*))} \le c \lVert \dot{\bfu}^* (t) \rVert_\infty$, and the nodal values of $\dot{\bfu}^*(t)$ are exactly the nodal values of $\mat u(\cdot,t)$. The smoothness assumption on $u$ and $\mat u$ thus implies $\lVert \dot{\bfu}^*\rVert_\infty \le c$, and we arrive at
	\begin{align}
	\left| \bfe_\bfu\T \left( \bfM(\bfx) - \bfM(\bfx^*) \right) \dot{\bfu}^* \right| &\le c \lVert \bfe_\bfu \rVert_{\bfM(\bfx^*)} \lVert \bfe_\bfx \rVert_{\bfA(\bfx^*)} \,.
	\end{align}
	Using the basis functions, Lemma~\ref{lemma: matrix difference} and the Leibniz formula, a tedious but elementary computation yields
	\begin{align}
	\bfe_\bfu\T \dfdt{}{t} \left( \bfM(\bfx) - \bfM(\bfx^*) \right) \bfu^* &= \int_0^1 \int_{\Omh^\theta} e_u^\theta \mat_h \nb \cdot e_x^\theta u_h^{*,\theta} \dx \dTheta \\
	&+ \int_0^1 \int_{\Omh^\theta} e_u^\theta \left(\nb \cdot e_x^\theta\right) u_h^{*,\theta} \nb \cdot v_h^\theta \dx \dTheta \,,
	\end{align}
	where $v_h^\theta$ is the velocity of $\Omh^\theta$ as a function of $t$, i.~e. the finite element function in $S_h(\bfx^*(t)+\theta \bfe_\bfx(t))$ with nodal vector $\dot{\bfx}^* + \theta \dot{\bfe}_\bfx = \bfv^* + \theta \bfe_\bfv$, implying $v_h^\theta = v_h^{*,\theta} + \theta e_v^\theta$.
	We will estimate both integrals separately, where we use the identity from Lemma~\ref{lemma: mat grad}. With $\mat_h e_x^\theta = e_v^\theta$ and writing $v_h^\theta = v_h^{*,\theta} + \theta e_v^\theta$, we obtain the following estimate for the first integral: (we write $L^p$ instead of $L^p(\Omh^\theta)$ and $\bfM$ and $\bfA$ instead of $\bfM(\bfx^*)$ and $\bfA(\bfx^*)$ in the occurring norms)
	\begin{align}
	\begin{aligned}
		&\left| \int_0^1 \int_{\Omh^\theta} e_u^\theta \mat_h \nb \cdot e_x^\theta u_h^{*,\theta} \dx \dTheta \right| \\
		&\le \int_0^1 \left\lVert e_u^\theta \right\rVert_{L^2} \left( \left\lVert \nb \cdot e_v^\theta \right\rVert_{L^2} + \left\lVert \nb v_h^{*,\theta} \right\rVert_{L^\infty} \left\lVert \nb e_x^\theta \right\rVert_{L^2} + \theta \left\lVert \nb e_v^\theta \right\rVert_{L^2} \left\lVert \nb e_x^\theta \right\rVert_{L^\infty} \right) \left\lVert u_h^{*,\theta} \right\rVert_{L^\infty} \\
		&\le c \left\lVert \bfe_\bfu \right\rVert_{L^2} \left( \lVert \nb e_v \rVert_{L^2} + \lVert \nb v_h^* \rVert_{L^\infty} \lVert \nb e_x \rVert_{L^2} + \lVert \nb e_v \rVert_{L^2} \lVert \nb e_x \rVert_{L^\infty} \right) \lVert u_h^* \rVert_{L^\infty} \\
		&\le c \lVert \bfe_\bfu \rVert_{\bfM} \left( \lVert \bfe_\bfv \rVert_{\bfA} + \lVert \nb v_h^* \rVert_{L^\infty} \lVert \bfe_\bfx \rVert_{\bfA} + \lVert \bfe_\bfv \rVert_{\bfA} \lVert \nb e_x \rVert_{L^\infty} \right) \lVert \bfu^* \rVert_{\infty} \\
		&\le c \lVert \bfe_\bfu \rVert_{\bfM(\bfx^*)} \left( \lVert \bfe_\bfv \rVert_{\bfA(\bfx^*)} + \lVert \bfe_\bfx \rVert_{\bfA(\bfx^*)} \right) \,.
	\end{aligned}
	\end{align}
	We analogously estimate the second integral and obtain
	\begin{align}
	\left| \int_0^1 \int_{\Omh^\theta} e_u^\theta \left(\nb \cdot e_x^\theta\right) u_h^{*,\theta} \nb \cdot v_h^\theta \dx \dTheta \right| \le c \lVert \bfe_\bfu \rVert_{\bfM(\bfx^*)} \left( \lVert \bfe_\bfv \rVert_{\bfA(\bfx^*)} + \lVert \bfe_\bfx \rVert_{\bfA(\bfx^*)} \right) \,.
	\end{align}
	
	Finally, we obtain for the first term of \eqref{eq: aux 3}:
	\begin{align}\label{eq: aux 6}
	- \bfe_\bfu\T \dfdt{}{t} \left( (\bfM(\bfx)-\bfM(\bfx^*)) \bfu^* \right) \le c \left\lVert \bfe_\bfu \right\rVert_{\bfM(\bfx^*)} \left( \left\lVert \bfe_\bfv \right\rVert_{\bfA(\bfx^*)} + \left\lVert \bfe_\bfx \right\rVert_{\bfA(\bfx^*)} \right) \,.
	\end{align}
	(ii) For the second term of \eqref{eq: aux 3}, we obtain similarly
	\begin{align}\label{eq: aux 7}
	\begin{aligned}
		&- \bfe_\bfu\T \dfdt{}{t} \left( (\bfM(\bfx)-\bfM(\bfx^*))\bfe_\bfu \right) \\
		&= -\frac{1}{2} \bfe_\bfu\T \left( \dfdt{}{t} (\bfM(\bfx)-\bfM(\bfx^*))\right) \bfe_\bfu - \frac{1}{2} \dfdt{}{t} \left( \bfe_\bfu\T (\bfM(\bfx)-\bfM(\bfx^*) \bfe_\bfu \right) \\
		&\le c \lVert \bfe_\bfu \rVert_{\bfM(\bfx^*)} \left( \lVert \bfe_\bfv \rVert_{\bfA(\bfx^*)} + \lVert \bfe_\bfx \rVert_{\bfA(\bfx^*)} \right) \lVert e_u \rVert_{L^\infty(\Omh(\bfx^*))} - \frac{1}{2} \dfdt{}{t} \left( \bfe_\bfu\T (\bfM(\bfx)-\bfM(\bfx^*)) \bfe_\bfu \right) \\
		&\le C \lVert \bfe_\bfu \rVert_{\bfM(\bfx^*)} \left( \lVert \bfe_\bfv \rVert_{\bfA(\bfx^*)} + \lVert \bfe_\bfx \rVert_{\bfA(\bfx^*)} \right) - \frac{1}{2} \dfdt{}{t} \left( \bfe_\bfu\T (\bfM(\bfx)-\bfM(\bfx^*)) \bfe_\bfu \right) \,.
	\end{aligned}
	\end{align}
	\erev
	(iii) For the third term, we use Lemma~\ref{lemma: matrix difference} and Lemma~\ref{lemma: aux 2} and estimate
	\begin{align}\label{eq: aux 8}
	\begin{aligned}
		\left| \bfe_\bfu\T (\bfA(\bfx)-\bfA(\bfx^*)) \bfu^* \right| &\le c \int_0^1 \lVert \nb e_u^\theta \rVert_{L^2(\Omh^\theta)} \lVert \nb e_x^\theta \rVert_{L^2(\Omh^\theta)} \lVert \nb u_h^{*,\theta} \rVert_{L^\infty(\Omh^\theta)} \dx \dTheta \\
		&\le c \lVert \nb e_u \rVert_{L^2(\Omh(\bfx^*))} \lVert \nb e_x \rVert_{L^2(\Omh(\bfx^*))} \lVert \nb u_h^* \rVert_{L^\infty(\Omh(\bfx^*))} \\
		&\le C \lVert \bfe_\bfu \rVert_{\bfA(\bfx^*)} \lVert \bfe_\bfx \rVert_{\bfA(\bfx^*)} \,,
	\end{aligned}
	\end{align}
	where we have used the smoothness assumption on $u$. 
	
	(iv) Similarly, we estimate 
	\begin{align}\label{eq: aux 9}
	\begin{aligned}
	\left| \bfe_\bfu\T (\bfA(\bfx)-\bfA(\bfx^*)) \bfe_\bfu \right| &= \left| \int_0^1 \int_{\Omh^\theta} \nb e_u^\theta \left( D_{\Omh^\theta} e_x^\theta \right) \nb e_u^\theta \dx \dTheta \right| \\
	&\le c \lVert \nb e_u \rVert_{L^2(\Omh(\bfx^*))}^2 \lVert \nb e_x \rVert_{L^\infty(\Omh(\bfx^*))} \\
	&\le c h^{(\kappa-3/2)/2} \lVert \bfe_\bfu \rVert_{\bfA(\bfx^*)}^2 \,.
	\end{aligned}
	\end{align}
	
	(v) For the fifth term, we use the Leibniz formula, an $L^\infty$-$L^2$-$L^2$-estimate and Lemma~\ref{lemma: aux 2} to obtain
	\begin{align}\label{eq: aux 10}
	\begin{aligned}
	\bfe_\bfu\T (\bff(\bfx)-\bff(\bfx^*)) &= \int_{\Omh^1} f e_u^1 \dx - \int_{\Omh^0} f e_u^0 \dx =\int_0^1 \dfdt{}{\theta} \int_{\Omh^\theta} f e_u^\theta \dx \dTheta \\
	&=\int_0^1 \int_{\Omh^\theta} \mat_\theta f e_u^\theta + f \underbrace{\mat_\theta e_u^\theta}_{=0} + f e_u^\theta \nb \cdot e_x^\theta \dx \dTheta \\
	&=\int_0^1 \int_{\Omh^\theta} f' e_x^\theta e_u^\theta + f e_u^\theta \nb \cdot e_x^\theta \dx \dTheta \\
	&\le \int_0^1 \lVert f' \rVert_{L^\infty(\Omh^\theta)} \lVert e_x^\theta \rVert_{L^2(\Omh^\theta)} \lVert e_u^\theta \rVert_{L^2(\Omh^\theta)} + \lVert f \rVert_{L^\infty(\Omh^\theta)} \lVert e_u^\theta \rVert_{L^2(\Omh^\theta)} \lVert \nb \cdot e_x^\theta \rVert_{L^2(\Omh^\theta)} \dTheta \\
	&\le c \lVert \bfe_\bfx \rVert_{\bfM(\bfx^*)} \lVert \bfe_\bfu \rVert_{\bfM(\bfx^*)} + c \lVert \bfe_\bfx \rVert_{\bfA(\bfx^*)} \lVert \bfe_\bfu \rVert_{\bfM(\bfx^*)} \\
	&\le c \lVert \bfe_\bfx \rVert_{\bfA(\bfx^*)} \lVert \bfe_\bfu \rVert_{\bfM(\bfx^*)} \,,
	\end{aligned}
	\end{align}
	where we have used the Poincar\'{e} inequality in the last step, which yields for $e_x^\ell \in H_0^1(\Omt)$
	\[ \lVert \bfe_\bfx \rVert_{\bfM(\bfx^*)} = \lVert e_x \rVert_{L^2(\Omh(\bfx^*(t)))} \le c \lVert e_x^\ell \rVert_{L^2(\Om(t))} \le c \lVert \nb e_x^\ell \rVert_{L^2(\Om(t))} \le c \lVert \bfe_\bfx \rVert_{\bfA(\bfx^*)} \,.\]
	
	(vi) For the last term of \eqref{eq: aux 3}, we use
	\begin{align}\label{eq: aux 11}
		\begin{aligned}
			\bfe_\bfu\T \bfM(\bfx^*) \bfd_\bfu &= \bfe_\bfu\T \bfK(\bfx^*)^{\frac{1}{2}} \bfK(\bfx^*)^{-\frac{1}{2}} \bfM(\bfx^*) \bfd_\bfu \\
			&\le \frac{1}{6} \lVert \bfK(\bfx^*)^{\frac{1}{2}} \bfe_\bfu \rVert_2^2 + C \lVert \bfK(\bfx^*)^{-\frac{1}{2}} \bfM(\bfx^*) \bfd_\bfu \rVert_2^2 \\
			&=\frac{1}{6} \lVert \bfe_\bfu \rVert_{\bfM(\bfx^*)}^2 + \frac{1}{6} \lVert \bfe_\bfu \rVert_{\bfA(\bfx^*)}^2 + C \lVert \bfd_\bfu \rVert_{\star,\bfx^*}^2 \,.
		\end{aligned}
	\end{align}
	Combining estimates (i)-(vi), using Young's inequality on each product, for $h \le h_0$ sufficiently small such that $c h^{(\kappa-3/2)/2} \le 1/6$, we obtain after absorbing $\lVert \bfe_\bfu \rVert_{\bfA(\bfx^*)}^2$:
	\begin{align}\label{eq: aux 12}
		\begin{aligned}
			\frac{1}{2} \dfdt{}{t} \lVert \bfe_\bfu \rVert_{\bfM(\bfx^*)}^2 + \frac{1}{2} \lVert \bfe_\bfu \rVert_{\bfA(\bfx^*)}^2 &\le c \lVert \bfe_\bfu \rVert_{\bfM(\bfx^*)}^2 + c \lVert \bfe_\bfx \rVert_{\bfA(\bfx^*)}^2 + c \lVert \bfe_\bfv \rVert_{\bfA(\bfx^*)}^2 \\
			&- \frac{1}{2} \dfdt{}{t} \left( \bfe_\bfu\T(\bfM(\bfx)-\bfM(\bfx^*)) \bfe_\bfu \right) + c \lVert \bfd_\bfu \rVert_{\star,\bfx^*}^2 \,.
		\end{aligned}
	\end{align}
	Inserting the estimates from Lemma~\ref{lem: stability}, we have
	\begin{align}\label{eq: aux 16}
	\begin{aligned}
	\dfdt{}{t} \lVert \bfe_\bfu \rVert_{\bfM(\bfx^*)}^2 + \lVert \bfe_\bfu \rVert_{\bfA(\bfx^*)}^2 &\le c \lVert \bfe_\bfu \rVert_{\bfM(\bfx^*)}^2 + c \lVert \bfd_\bfv \rVert_{\star,\bfx^*}^2 + c \int_0^t \lVert \bfd_\bfv(s) \rVert_{\star,\bfx^*(s)}^2 \mathrm{d} s \\
	&- \dfdt{}{t} \left( \bfe_\bfu\T (\bfM(\bfx)-\bfM(\bfx^*)) \bfe_\bfu \right) + c \lVert \bfd_\bfu \rVert_{\star,\bfx^*}^2 \,.
	\end{aligned}
	\end{align}
	Integrating from $0$ to $t$ for $t \in [0,t^*]$ and using a Gronwall argument as in part (C) of \cite[Proposition 6.1]{KLLP17}, we finally obtain the desired result for $t \in [0,t^*]$. The proof is then finished by showing that $t^*$ coincides with $T$, which is due to the same argument as in the previous section.
\end{proof}

\brev
\begin{remark}
	The previous lemma remains valid in the two-dimensional case, where the assumption~\eqref{eq: defect assumption2} is only required for $\kappa > 1$.
\end{remark}
\erev

\section{Defect bounds}\label{section: defect bounds}
In this section we show that the smallness assumptions in Lemma~\ref{lem: stability} and Lemma~\ref{lemma: stability u} are satisfied for $\kappa = k \ge 2$, which in combination with the stability results will lead to the desired error bounds. We remind that we have different dual norm definitions \eqref{eq: dual norm} and \eqref{eq: dual norm u} since the defect functions live in different finite element spaces. We avoid using different notations, because the dual norms only appear on $\bfd_\bfv$ and $\bfd_\bfu$, so it is always clear from context which definition is meant.

\subsection{The interpolating domain}
In order to estimate the defect $\bfd_\bfu$, we need to introduce a discrete velocity on the smooth domain, which is denoted by $\widehat{v}_h$.

\brev Recall \erev that $\Omt$ can be described as image $X(\cdot,t)(\Om_0)$ with a sufficiently smooth map $X: \Om_0 \times [0,T] \to \R^3$. The nodes $x_j^*(t) = X(x_j^0,t)$ define an interpolating domain which is parametrized over $\Omh^0$ via
\[ X_h^*(p_h,t) = \sum_{j=1}^N x_j^*(t) \varphi_j[\bfx(0)](p_h) \,,\quad p_h \in \Omh^0 \,.\]
The velocity of the interpolating domain is given, \brev using the transport property of the basis functions \eqref{transport property}, by
\[ v_h^*(\cdot,t) = \sum_{j=1}^N v_j^*(t) \varphi_j[\bfx^*(t)](\cdot) \quad \text{with }v_j^*(t) = \dfdt{}{t} x_j^*(t) \,. \]

For a material point $p_h(t) = X_h^*(p_h,t) \in \Omh(\bfx^*(t))$, $p_h \in \Omh^0$, on the interpolated exact domain, \brev this velocity satisfies \erev
\[ v_h^*(p_h(t),t) = \dfdt{}{t} X_h^*(p_h,t) \,. \]
Associated with $p_h(t)$ is its lifted material point $y(t) = \Lambda_h(p_h(t),t) \in \Omt$. This lifted point moves with velocity
\[ \widehat{v}_h(y(t),t) = \dfdt{}{t} y(t) = \dfdt{}{t} \Lambda_h(p_h(t),t) = (\partial_t \Lambda_h)(p_h(t),t) + v_h^*(p_h(t),t) \nb \Lambda_h(p_h(t),t) \,. \]
We can use these velocities to define discrete material derivatives for functions $\varphi_h$ and $\varphi$ defined on $\Omh(\bfx^*(t))$ and $\Omt$, respectively, via 
\begin{align}\label{def: discrete material derivatives}
	\mat_{v_h^*} \varphi_h &= \partial_t \varphi_h + v_h^* \cdot \nb \varphi_h \,,\\
	\mat_{\widehat{v}_h} \varphi &= \partial_t \varphi + \widehat{v}_h \cdot \nb \varphi \,.	
\end{align}
The basis functions $\varphi_j[\bfx^*]$ enjoy the transport property $\mat_{v_h^*} \varphi_j = 0$. It is not true in general that the lifted basis functions \brev satisfy $\mat_v \varphi_j^\ell = 0$, with $\mat_v = \mat$ as defined in~\eqref{eq: material derivative}. In particular, we have $(\mat_{v_h^*} \varphi_j)^\ell \ne \mat \varphi_j^\ell$ in general. \erev The following lemma shows that the transport property is satisfied with the discrete velocity defined above, which will be crucial in the following.
\begin{lemma}\label{lemma: transport property discrete lifted}
	For $j=1,\ldots,N$, we have
	\[ \mat_{\widehat{v}_h} \varphi_j^\ell = 0 \,. \]
	In particular, we have for any finite element function $\eta_h \in S_h(\bfx^*(t))$ and for any $u \in H^{k+1}(\Omt)$
	\[ \left( \mat_{v_h^*} \eta_h \right)^\ell = \mat_{\widehat{v}_h} \eta_h^\ell \quad\text{and}\quad  \left( \mat_{v_h^*} \widetilde{I}_h u \right)^\ell = \mat_{\widehat{v}_h} I_h u = I_h \mat_{\widehat{v}_h} u \,. \]
\end{lemma}
\begin{proof}\noindent
	\brev Follows from Definition~\ref{def: lift}, the chain rule and the transport property of the basis functions, cf.  \cite[Lemma 4.1]{DE13a}. \erev
\end{proof}

For the following defect estimate, we introduce the notation
\begin{align}
	q_h^*(\eta_h,\chi_h) &= \int_{\Omh^*(t)} \eta_h \chi_h \nb \cdot v_h^* \dx \,,\\
	\widehat{q}_h(\eta,\chi) &= \int_{\Omt} \eta \chi \nb \cdot \widehat{v}_h \dx \,.
\end{align}
\begin{lemma}\label{lemma: transport properties}
	For any $\eta(\cdot,t)$, $\chi(\cdot,t) \in H^1(\Omt)$, we have
	\begin{align}
		\dfdt{}{t} m(\eta,\chi) &= m(\mat \eta,\chi) + m (\eta,\mat \chi) + q(\eta,\chi) \,,\\
		\dfdt{}{t} m(\eta,\chi) &= m\left(\mat_{\widehat{v}_h} \eta,\chi\right) + m\left(\eta, \mat_{\widehat{v}_h} \chi\right) + \widehat{q}_h(\eta,\chi) \,.
	\end{align}
	On the discrete domain, for $\eta_h(\cdot,t)$, $\chi_h(\cdot,t) \in S_h(\Omh(\bfx^*(t)))$, we have
	\begin{align}
	\dfdt{}{t} m_h^*(\eta_h,\chi_h) = m_h^*\left( \mat_{v_h^*} \eta_h , \chi_h \right) + m_h^*\left( \eta_h, \mat_{v_h^*} \chi_h \right) + q_h^*(\eta_h,\chi_h) \,.
	\end{align}
\end{lemma}
\begin{proof}
	Follows directly from the Leibniz formula (see \cite[Lemma 7.12]{ER17pre}).
\end{proof}

We are now in position to formulate and prove the required defect estimates.
\begin{lemma}\label{lemma: defect estimates}
	Let the domain $\Omt$ and the exact solution $(u,v,X)$ be sufficiently smooth. Then there is a constant $c > 0$ and an $h_0 > 0$, such that for all $h \le h_0$ and all $t \in [0,T]$, the defects $\bfd_{\bfv^\Om}$ and $\bfd_\bfu$ are bounded by
	\begin{align}
	\lVert \bfd_{\bfv^\Om} \rVert_{\star,\bfx^*} &\le c h^k \,, \\
	\lVert \bfd_\bfu \rVert_{\star,\bfx^*} &\le c h^k \,.
	\end{align}
\end{lemma}
\begin{proof}\noindent
	We start with estimating $\bfd_\bfu$. The defect equation \eqref{eq: defect du} is equivalent to
	\begin{align}
		m_h^*(d_u,\varphi_h) &= \dfdt{}{t} m_h^*(\widetilde{I}_h u,\varphi_h) + a_h^* (\widetilde{I}_h u, \varphi_h) - m_h^*(f,\varphi_h) \\
		&= m_h^*\left( \mat_{v_h^*} \widetilde{I}_h u, \varphi_h \right) + q_h^*(\widetilde{I}_h u,\varphi_h) + a_h^*(\widetilde{I}_h u, \varphi_h ) - m_h^*(f,\varphi_h)
	\end{align}
	for all $\varphi_h \in S_h(\bfx^*)$. The exact solution $u$ satisfies, using Lemma~\ref{lemma: transport properties} and Lemma~\ref{lemma: transport property discrete lifted},
	\begin{align}
		0 &= \dfdt{}{t} m(u,\varphi_h^\ell) + a(u,\varphi_h^\ell) - m(f,\varphi_h^\ell) \\
		&=m\left( \mat_{\widehat{v}_h} u, \varphi_h^\ell \right) + \widehat{q}_h(u,\varphi_h^\ell) + a(u,\varphi_h^\ell) - m(f,\varphi_h^\ell) \,.
	\end{align}
	Subtracting both terms yields
	\begin{align}\label{eq: defect u}
		m_h^*(d_u,\varphi_h) &= \left( m_h^*\left( \mat_{v_h^*} \widetilde{I}_h u, \varphi_h \right) - m\left( \mat_{\widehat{v}_h} u, \varphi_h^\ell \right) \right) + \left( q_h^*( \widetilde{I}_h u,\varphi_h) - \widehat{q}_h(u,\varphi_h^\ell) \right) \\
		&+\left( a_h^*(\widetilde{I}_h u,\varphi_h) - a(u,\varphi_h^\ell) \right) - \left( (m_h^*(f,\varphi_h) - m(f,\varphi_h^\ell) ) \right) \,.
	\end{align}
	We will estimate the four differences separately.
	
	(i) For the first difference, we use $\mat_{\widehat{v}_h} I_h u = I_h \mat_{\widehat{v}_h} u$:
	\begin{align}
		\left| m_h^* \left( \mat_{v_h^*} \widetilde{I}_h u, \varphi_h \right) - m\left( \mat_{\widehat{v}_h} u, \varphi_h^\ell \right) \right| &\le \left| m_h^* \left( \mat_{v_h^*} \widetilde{I}_h u, \varphi_h \right) - m\left( \mat_{\widehat{v}_h} I_h u, \varphi_h^\ell \right) \right| +\left| m\left( I_h \mat_{\widehat{v}_h} u - \mat_{\widehat{v}_h}u,\varphi_h^\ell \right) \right| \,.
	\end{align}
	For the first term, note that $(\mat_{v_h^*} \widetilde{I}_h u)^\ell = \mat_{\widehat{v}_h} I_h u$, so Lemma~\ref{lemma: error bilinear forms} yields
	\begin{align}
		\left| m_h^* \left( \mat_{v_h^*} \widetilde{I}_h u, \varphi_h \right) - m\left( \mat_{\widehat{v}_h} I_h u, \varphi_h^\ell \right) \right| &\le c h^k \left\lVert \mat_{\widehat{v}_h} I_h u \right\rVert_{L^2(\Om)} \lVert \varphi_h^\ell \rVert_{L^2(\Om)} \,.
	\end{align}
	Now we bound
	\begin{align}
		\left\lVert \mat_{\widehat{v}_h} I_h u \right\rVert_{L^2(\Om)} &= \left\lVert I_h \mat_{\widehat{v}_h} u - \mat_{\widehat{v}_h} u + \mat_{\widehat{v}_h} u \right\rVert_{L^2(\Om)} \le (c h^k + 1) \left\lVert \mat_{\widehat{v}_h} u - \mat u + \mat u \right\rVert_{L^2(\Om)} \le c \,,
	\end{align}
	where we have used that $\lVert \mat_{\widehat{v}_h}u-\mat u \rVert_{L^2(\Om)} \le ch^{k+1}$ (see \cite[Lemma 7.14]{ER17pre}) and the regularity assumption on $u$. Similarly
	\begin{align}
		\left| m\left( I_h \mat_{\widehat{v}_h} u - \mat_{\widehat{v}_h} u, \varphi_h^\ell \right) \right| &\le \left\lVert I_h \mat_{\widehat{v}_h} u - \mat_{\widehat{v}_h} u \right\rVert_{L^2(\Om)} \left\lVert \varphi_h^\ell \right\rVert_{L^2(\Om)} \le c h^k \left\lVert \mat_{\widehat{v}_h} u \right\rVert \left\lVert \varphi_h^\ell \right\rVert_{L^2(\Om)} \le c h^k  \left\lVert \varphi_h^\ell \right\rVert_{L^2(\Om)} \,.
	\end{align}
	Altogether, we have for the first difference of \eqref{eq: defect u}
	\begin{align}
		\left| m_h^* \left( \mat_{v_h^*} \widetilde{I}_h u, \varphi_h \right) - m\left( \mat_{\widehat{v}_h} u, \varphi_h^\ell \right) \right| &\le c h^k \left\lVert \varphi_h^\ell \right\rVert_{L^2(\Om)} \,.
	\end{align}
	
	(ii) In a similar way:
	\begin{align*}
		\left| q_h^*(\widetilde{I}_h u, \varphi_h) - \widehat{q}_h(u,\varphi_h^\ell)\right| &\le \left| q_h^*(\widetilde{I}_h u, \varphi_h) - \widehat{q}_h(I_h u, \varphi_h^\ell)\right| + \left| \widehat{q}_h(I_h u -u, \varphi_h^\ell ) \right| \,.
	\end{align*}
	For the first term, we use \cite[Lemma 7.15]{ER17pre}:
	\begin{align*}
		\left| q_h^*(\widetilde{I}_h u, \varphi_h) - \widehat{q}_h(I_h u, \varphi_h^\ell)\right| \le c h^{k+1} \left\lVert I_h u \right\rVert_{L^2(\Om)} \left\lVert \varphi_h^\ell \right\rVert \le c h^{k+1} \left\lVert \varphi_h^\ell \right\rVert_{L^2(\Om)} \,.
	\end{align*}
	For the second term, we use an $L^2$-$L^2$-$L^\infty$ estimate and \cite[Lemma 7.14]{ER17pre} to bound $\lVert \nb \widehat{v}_h \rVert_{L^\infty(\Om)}$:
	\begin{align*}
		\left| \widehat{q}_h(I_h u -u, \varphi_h^\ell ) \right| &\le \lVert I_h u - u \rVert_{L^2(\Om)} \lVert \varphi_h^\ell \rVert_{L^2(\Om)} \lVert \nb \cdot \widehat{v}_h \rVert_{L^\infty(\Om)} \\
		&\le c h^{k+1} \lVert \varphi_h^\ell \rVert_{L^2(\Om)} \lVert \nb \widehat{v}_h \rVert_{L^\infty(\Om)} \le c h^{k+1} \lVert \varphi_h^\ell \rVert_{L^2(\Om)} \,.
	\end{align*}
	
	(iii) The third term of \eqref{eq: defect u} is estimated similarly:
	\begin{align*}
		\left| a_h^*(\widetilde{I}_h u, \varphi_h) - a(u,\varphi_h^\ell) \right| &\le \left| a_h^*(\widetilde{I}_h u,\varphi_h ) - a(I_h u, \varphi_h^\ell) \right| + \left| a(I_h u - u, \varphi_h^\ell ) \right| \\
		&\le c h^k \lVert \nb I_h u \rVert_{L^2(\Om)} \lVert \nb \varphi_h^\ell \rVert_{L^2(\Om)} + \lVert \nb (I_h u - u) \rVert_{L^2(\Om)} \lVert \nb \varphi_h^\ell \rVert_{L^2(\Om)} \\
		&\le c h^k \lVert \varphi_h^\ell \rVert_{H^1(\Om)} \,.
	\end{align*}
	
	(iv) For the last term of \eqref{eq: defect u}, we immediately have
	\[ \left| m_h^*(f,\varphi_h) - m(f,\varphi_h^\ell) \right| \le c h^k \lVert f \rVert_{L^2(\Om)} \lVert \varphi_h^\ell \rVert_{L^2(\Om)} \,. \]
	
	Putting those four estimates together, using norm equivalence, we obtain
	\[ \lVert \bfd_\bfu \rVert_{\star,\bfx^*} = \lVert d_u \rVert_{H_h^{-1}(\Om(\bfx^*))} = \sup_{0 \ne \varphi_h \in S_h(\bfx^*)} \frac{m_h^*(d_u,\varphi_h)}{\lVert \varphi_h \rVert_{H^1(\Om(\bfx^*))}} \le c h^k \,. \]
	
	Now we estimate $\bfd_{\bfv^\Om}$, which is defined by the defect equation \eqref{eq: ODE system defects}. We set $\bfd_\bfv\T = (0,\bfd_{\bfv^\Om}\T)$ and $\bfw\T = (0,\bfw^{\Om,\mathrm{T}})$ for $\bfw^\Om \in \R^{3 N_\Om}$ and test with $\bfw^\Om$ to obtain with a computation similar to \eqref{eq: aux combine} (omitting the tensor notation) $\bfw\T \bfM(\bfx^*) \bfd_\bfv =\bfw\T \bfA(\bfx^*) \bfv^*$ which is equivalent to
	\begin{align}\label{eq: defect v}
		\int_{\Omh(\bfx^*)} \varphi_h \cdot d_h \dx &= \int_{\Omh(\bfx^*)} \nb \varphi_h \cdot \nb v_h^* \dx =a_h(\varphi_h, \widetilde{I_h} v) - a(\varphi_h^\ell, I_h v) + a(\varphi_h^\ell, I_h v)
	\end{align}
	for all $\varphi_h \in S_{0,h}(\bfx^*)$. We will estimate the first difference and the second term of \eqref{eq: defect v} separately, starting with the second term. Since $\varphi_h \in S_{0,h}(\bfx^*)$, we have $\varphi_h^\ell \in H_0^1(\Omt)$ and thus $a(\varphi_h^\ell,v) = 0$. With Proposition~\ref{lemma: interpolation}, we obtain
	\begin{align*}
		a(\varphi_h^\ell,I_h v) &= a(\varphi_h^\ell,I_h v - v) \le \lVert \nb \varphi_h^\ell \rVert_{L^2(\Om)} \lVert \nb (I_h v - v) \rVert_{L^2(\Om)} \le c h^k \lVert \nb \varphi_h^\ell \rVert_{L^2(\Om)} \lVert v \rVert_{H^{k+1}(\Om)} \,.
	\end{align*}
	The first difference in \eqref{eq: defect v} is estimated analogously to (iii) in the first part of this proof and yields
	\begin{align*}
		|a_h(\varphi_h,\widetilde{I_h} v) - a (\varphi_h^\ell,I_h v) | 
		\le  c h^k \lVert \nb \varphi_h^\ell \rVert_{L^2(\Om)}
	\end{align*}
	for $h \le h_0$ sufficiently small using the regularity assumption.

	Putting these estimates together yields with Lemma~\ref{lem: norm equivalency}:
	\[ \lVert \bfd_{\bfv^\Om} \rVert_{\star,\bfx^*} = \sup_{0 \ne \bfw^\Om \in \R^{3 N_\Om}} \frac{\bfd_{\bfv^\Om}\T \bfM_{22}(\bfx*) \bfw^\Om}{\lVert \bfw^\Om \rVert_{\bfA_{22}(\bfx^*)}} \le c h^k \,. \]
\end{proof}

\section{Proof of Theorem \ref{theorem: error bound}}\label{section: proof}
We prove the first error bound. The remaining ones are shown analogously. The error is decomposed using interpolation and lift:
\[ u_h^L - u = \left( \widehat{u}_h - \widetilde{I}_h u \right)^\ell + \left( I_h u - u \right) \,.\]
The right term can be bounded by $c h^k$ in the $H^1$-norm using an interpolation estimate. For the first term we obtain, using norm equivalence, Lemma~\ref{lem: stability} and Lemma~\ref{lemma: defect estimates}
\begin{align}
\left\lVert (\widehat{u}_h - \widetilde{I}_hu)^\ell \right\rVert_{L^2(\Omt)} &\le c \lVert \widehat{u}_h - \widetilde{I}_h u \rVert_{L^2(\Omh(\bfx^*(t)))} = c \lVert \bfe_\bfu \rVert_{\bfM(\bfx^*)} \\
&\le c \int_0^t \lVert \bfd_\bfu(s) \rVert_{\star,\bfx^*}^2 + \lVert \bfd_\bfv(s) \rVert_{\star,\bfx^*}^2 \mathrm{d}s \le c h^k \,.
\end{align}
Analogously
\[ \lVert \nb ( \widehat{u}_h - \widetilde{I}_h u )^\ell \rVert_{L^2(\Omt)} \le c \lVert \nb ( \widehat{u}_h - \widetilde{I}_h u ) \rVert_{L^2(\Omh(\bfx^*(t)))} = c \lVert \bfe_\bfu \rVert_{\bfA(\bfx^*)} \,.\]
Lemma~\ref{lemma: stability u} and Lemma~\ref{lemma: defect estimates} yield the result. The remaining estimates are shown analogously.
\brev
\begin{remark}\emph{($L^2$-estimate)}\\
	The convergence rate in $u, v$ and $X$ in the $L^2$-norm is expected to be of order $k+1$, which is also reflected in the numerical experiments down below. In order to prove $\mathcal{O}(h^{k+1})$-error bounds  for the diffusion equation, one could work with the Ritz projection $R_h u$ instead of the interpolation $I_h u$. In fact, defining a Ritz projection as described in \cite[Section 3.3.2]{ER17pre}, cf. \cite{DE13a}, we are able to prove $\sup_{t \in [0,T]} \lVert \bfd_\bfu(t) \rVert_{\star,\bfx^*(t)} \le c h^{k+1}$. This yields the error bound
	\begin{align}
		\lVert \bfe_\bfu(t) \rVert_{\bfM(\bfx^*)}^2 + \int_0^t \lVert \bfe_\bfu(s) \rVert_{\bfA(\bfx^*(s))}^2 \mathrm{d}s \le C h^{2k+2} + c \int_0^t \lVert \bfd_\bfv(s) \rVert_{\star,\bfx^*(s)}^2 \mathrm{d} s\,.
	\end{align}
	It is further possible to define a Ritz map for the Laplace equation for the velocity, taking the inhomogeneous boundary conditions into account. However, taking the Ritz projection instead of the finite element interpolation implies that the corresponding error $\bfe_\bfv$ does not vanish on the boundary anymore. This induces a different defect $\bfd_\bfv$ in $v$ and an additional defect $\bfd_\bfx$ in the equation $\dot{\bfe}_\bfx = \bfe_\bfv + \bfd_\bfx$, where $\bfd_\bfx$ can be considered as the error between the finite element interpolation $I_h v$ and the Ritz projection $R_h v$ of $v$. While it is indeed possible to obtain an $\mathcal{O}(h^{k+1})$ bound for $\bfd_\bfv$, this is no longer true for the new defect $\bfd_\bfx$, which still has to be estimated in the $\bfA$-norm, see \eqref{ineq: milestone 1}, yielding only a $h^k$ error bound.
\end{remark}
\erev

\section{Numerical experiments}\label{section: examples}
\brev In this section we illustrate the theoretical results with various numerical experiments. Fitting the layout of the stability proof, we start with an evolving domain problem in two dimensions without solving a diffusion equation on that domain. The second example is similar to the first one but three-dimensional. In the third example, we show convergence plots for a diffusion equation with non-homogeneous Neumann boundary conditions on a rotating and growing sphere.\erev

All experiments were implemented in MATLAB\textsuperscript{\textregistered} R2018a and performed in reasonable time on an MSI GE63VR notebook with Intel Core i7-7700HQ processor and 16 GB DDR4-RAM.

\subsection{An evolving open domain}\label{subsection: example 1}
We consider problem \eqref{eq: Poisson equation velocity} for $t \in [0,1]$, with $\Om(0)$ being the unit circle in $\R^2$. As exact solution, we choose 
\begin{align}
	v(x,t) &= \begin{pmatrix}
	\exp(-2t) \left( \exp(x_1) \sin(x_2) - \exp(x_2) \sin (x_1) \right) \\
	2 \exp(-5t) \left( x_1^2 - x_2^2 \right)
	\end{pmatrix} \,,
\end{align}
which satisfies $-\laplace v = 0$. Exemplary triangulations of $\Om(t_j)$ for $t_j=j/5$, $j=0,\ldots,5$ are shown in Figures~\ref{fig: triang1} and  \ref{fig: triang2}.
\begin{figure}[!b]
	\begin{minipage}{.3\textwidth}
		\includegraphics[clip=true,scale=0.25]{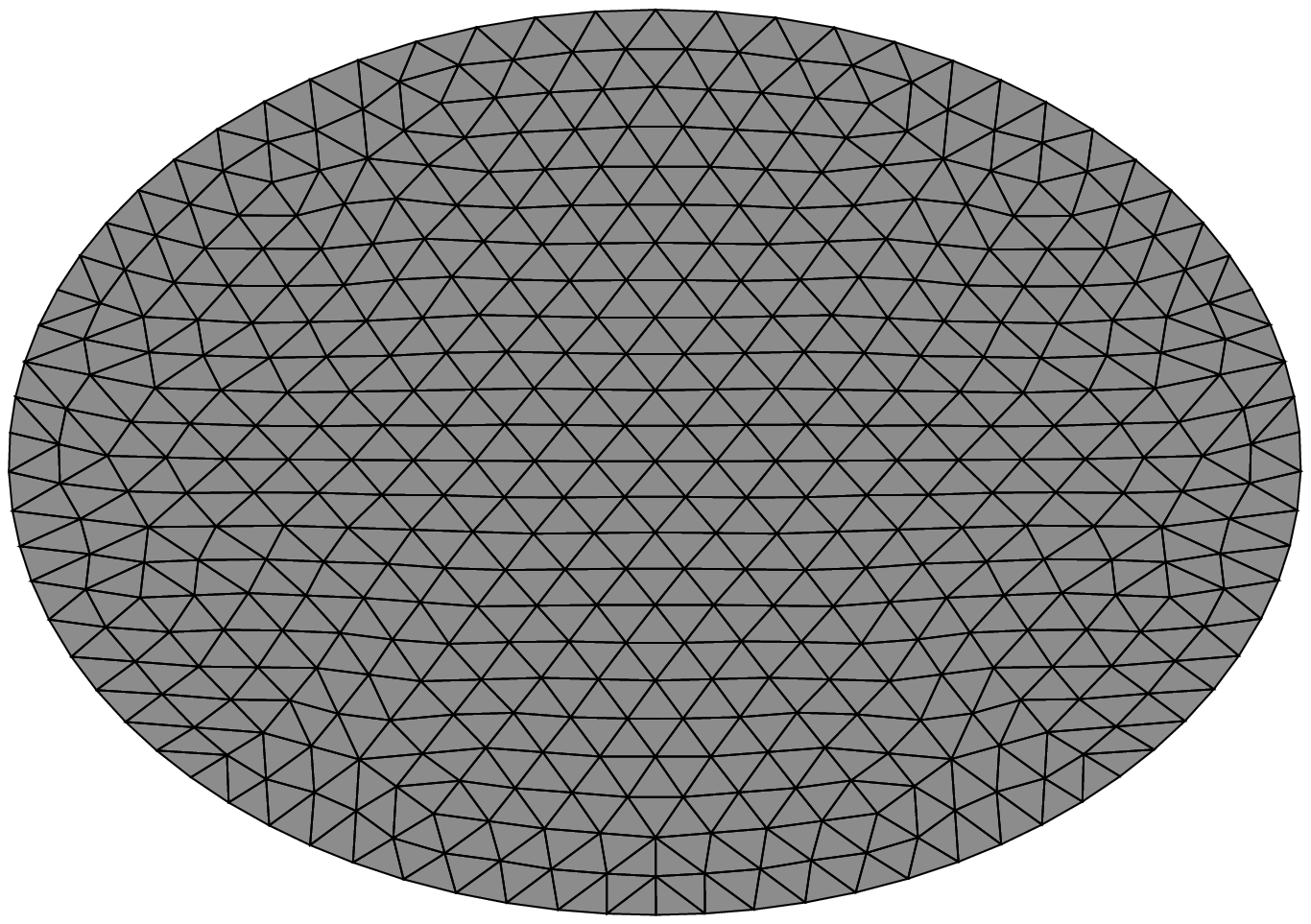}
	\end{minipage}	
	\begin{minipage}{.3\textwidth}
		\includegraphics[clip=true,scale=0.25]{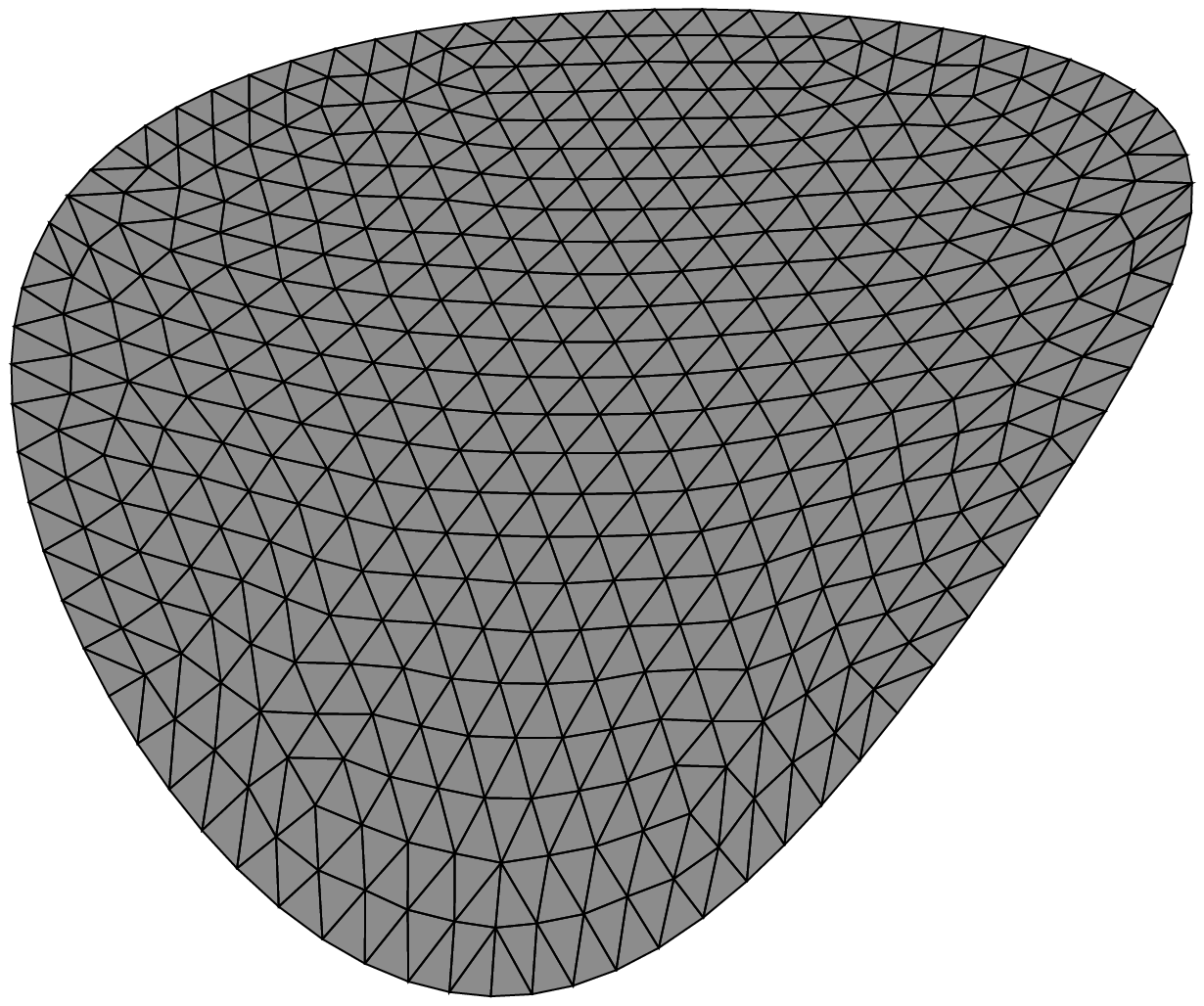}
	\end{minipage}
	\begin{minipage}{.3\textwidth}
		\includegraphics[clip=true,scale=0.25]{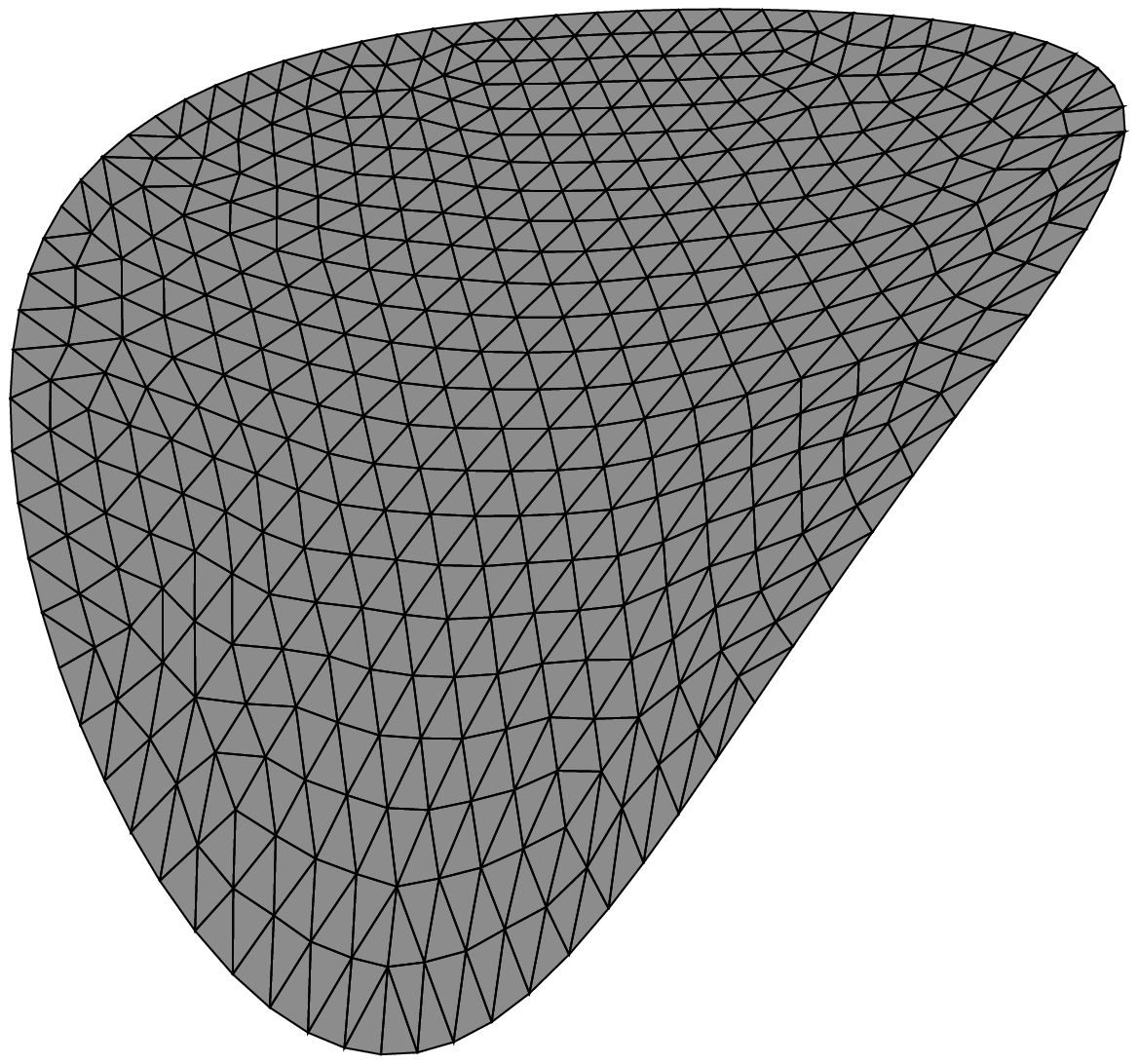}
	\end{minipage}
	\caption{Triangulation of $\Omt$ at $t_0=0$ (left), $t_1=0.2$ (center) and $t_2=0.4$ (right).}\label{fig: triang1}
\end{figure}
\begin{figure}[!t]
	\begin{minipage}{.30\textwidth}
		\includegraphics[clip=true,scale=0.25]{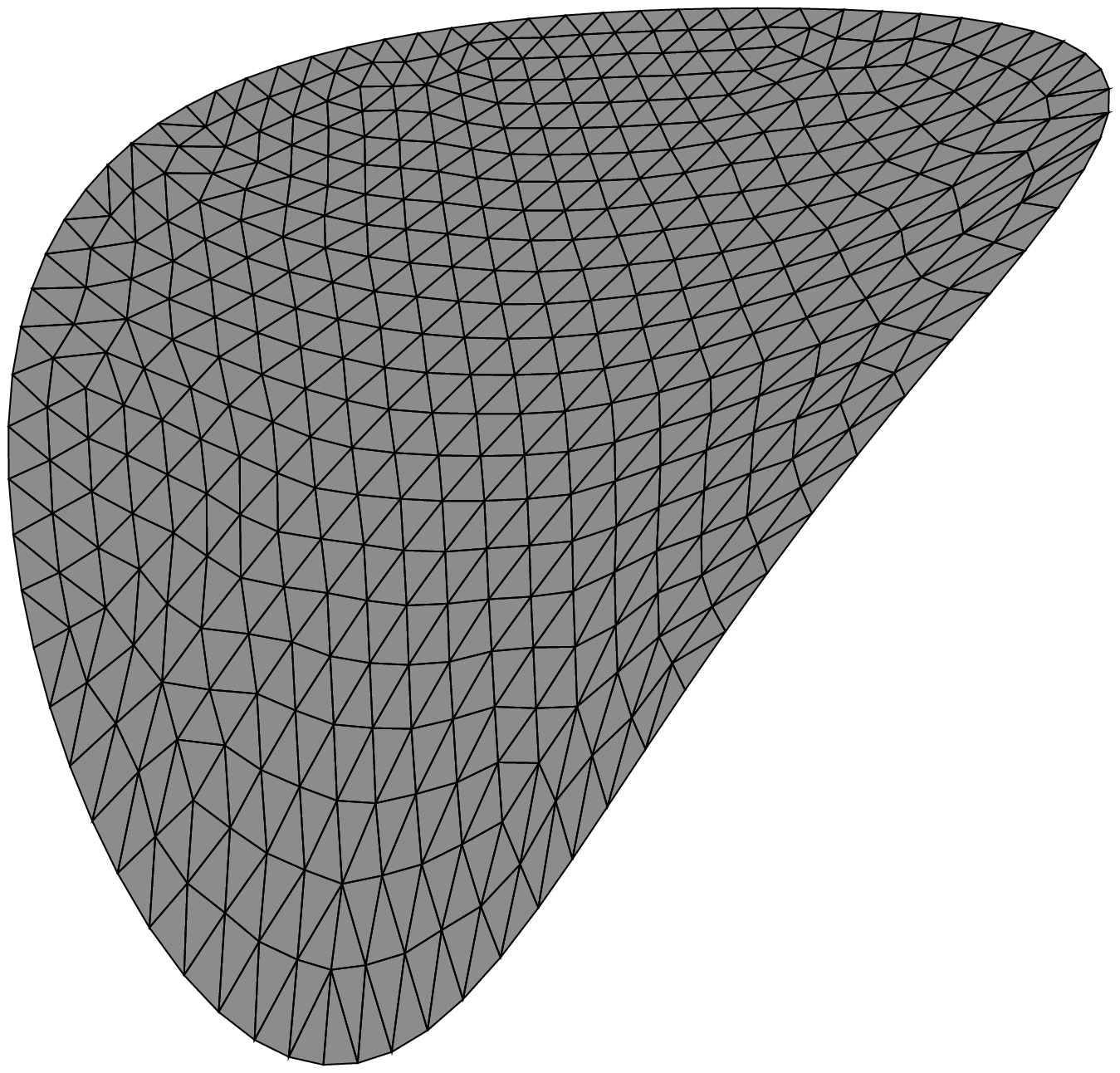}
	\end{minipage}	
	\begin{minipage}{.30\textwidth}
		\includegraphics[clip=true,scale=0.25]{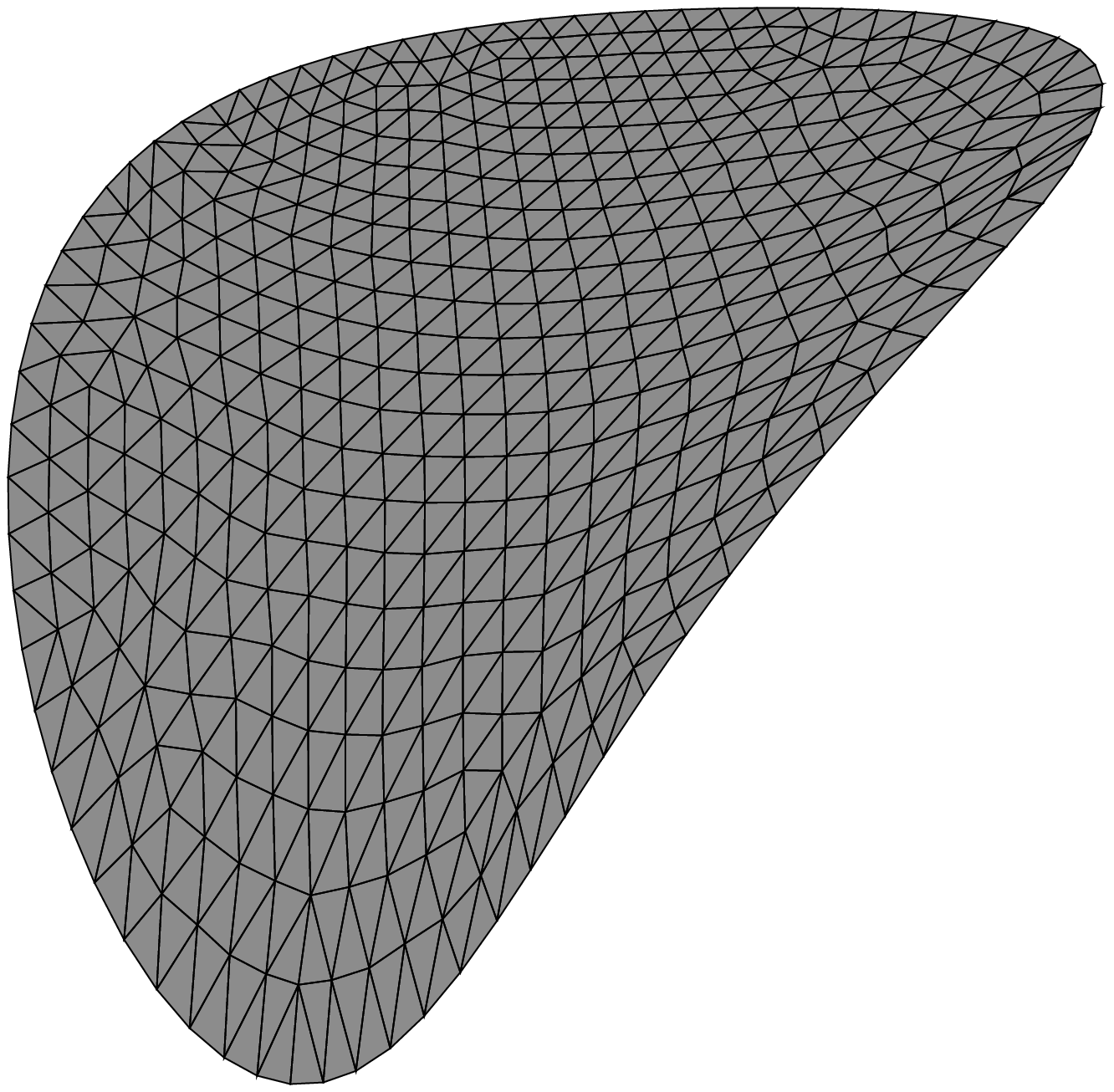}
	\end{minipage}
	\begin{minipage}{.30\textwidth}
		\includegraphics[clip=true,scale=0.25]{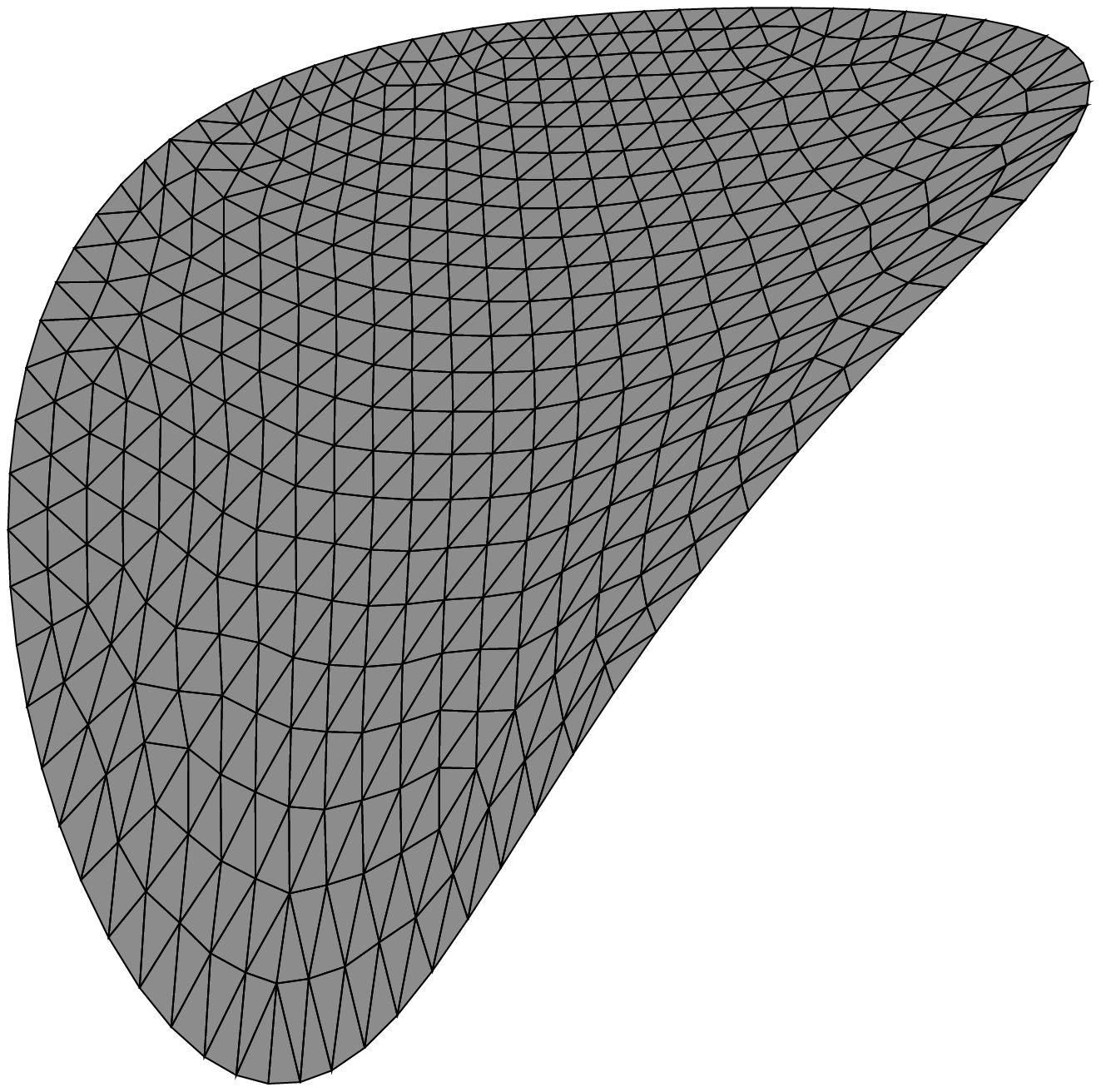}
	\end{minipage}
	\caption{Triangulation of $\Omt$ at $t_3=0.6$ (left), $t_4=0.8$ (center) and $t_5=1.0$ (right).}\label{fig: triang2}
\end{figure}

We apply a second order isoparametric finite element method. \brev For time discretization, we use a linearly implicit 4-step BDF method with time step size $\tau = 8\cdot10^{-3}$, such that the time discretization error is negligible\erev. To compute a reference solution, we use the fact that the above $v$ satisfies $-\laplace v(\cdot,t)=0$, and solve the position ODEs
\[ \dfdt{}{t} x_j(t) = v(x_j(t),t) \,,\quad x_j(0)=x_j^0 \,, \]
in \emph{all} nodes $x_j^0$, $j=1,\ldots,N$, of the initial triangulation with a RK4 method and time step size $\tau = 2 \cdot 10^{-4}$. Since stiffness is no issue in the position ODEs, an explicit high-order time discretization scheme is sufficient to compute a reference solution.

We record the position error
\begin{align}
	\lVert \mathrm{err}_\bfx \rVert_{L^\infty(L^2)} &:= \sup_{n: n \tau \le 1} \lVert (x_h^n)^L - \mathrm{id}_{\Om(t_n)} \rVert_{L^2(\Om(t_n))^2} \,,\\
	\lVert \mathrm{err}_\bfx \rVert_{L^\infty(H^1)} &:= \sup_{n: n \tau \le 1} \lVert \nb \left( (x_h^n)^L - \mathrm{id}_{\Om(t_n)} \right) \rVert_{L^2(\Om(t_n))^2} \,.
\end{align}
\brev and the velocity error $\mathrm{err}_\bfv$ in the same norms for different choices of $h$. Figure~\ref{fig: ex1_bdf3_spatial} shows the results. The error in $H^1$-norm converges with the expected order, whereas the convergence rate of the $L^2$-norm error is one order higher. This is not covered by the theory of this paper and left to possible future works. \erev
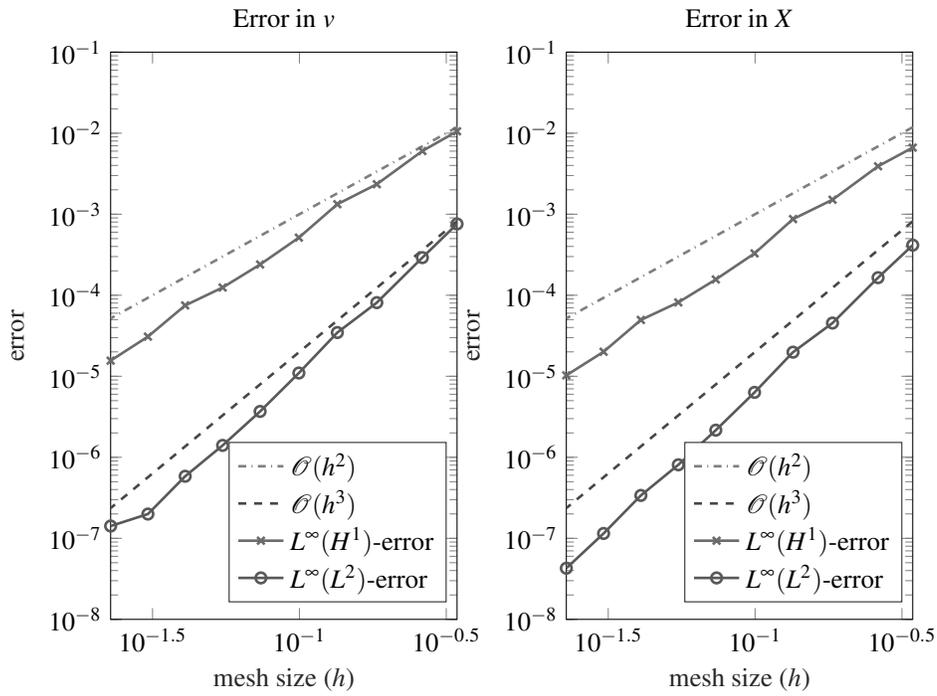
\begin{figure}[htb!]
	\begin{center}
%
%
\definecolor{mycolor1}{rgb}{0.5,0.5,0.5}%
\definecolor{mycolor2}{rgb}{0.28,0.28,0.28}%
\definecolor{mycolor3}{rgb}{0.42,0.42,0.42}%
\definecolor{mycolor4}{rgb}{0.35,0.35,0.35}%
\begin{tikzpicture}

\begin{axis}[%
width=1.813in,
height=2.971in,
at={(0.704in,0.401in)},
scale only axis,
xmode=log,
xmin=0.0227682496375,
xmax=0.344185326777365,
xminorticks=true,
xlabel style={font=\color{white!15!black}},
xlabel={mesh size ($h$)},
ymode=log,
ymin=1e-08,
ymax=0.1,
yminorticks=true,
ylabel style={font=\color{white!15!black}},
ylabel={error},
axis background/.style={fill=white},
title={Error in $v$},
legend style={at={(0.97,0.03)}, anchor=south east, legend cell align=left, align=left, draw=white!15!black}
]
\addplot [color=mycolor1, dashdotted, line width=1.0pt]
  table[row sep=crcr]{%
0.344185326777365	0.0118463539168842\\
0.262247567979484	0.00687737869111541\\
0.183511800711609	0.00336765810004173\\
0.134653571411827	0.001813158429396\\
0.0995855310120098	0.000991727798694396\\
0.0735033052111926	0.000540273587696974\\
0.054782539201058	0.000300112660131545\\
0.0408424807519244	0.000166810823397132\\
0.0304951795393483	9.29955975137088e-05\\
0.0227682496375	5.1839319155552e-05\\
};
\addlegendentry{$\mathcal{O}(h^2)$}

\addplot [color=mycolor2, dashed, line width=1.0pt]
  table[row sep=crcr]{%
0.344185326777365	0.000815468238800619\\
0.262247567979484	0.000360715167163788\\
0.183511800711609	0.000123601000423939\\
0.134653571411827	4.8829651610726e-05\\
0.0995855310120098	1.97523478904706e-05\\
0.0735033052111926	7.94237888280735e-06\\
0.054782539201058	3.28818671367803e-06\\
0.0408424807519244	1.362593568762e-06\\
0.0304951795393483	5.67183488509904e-07\\
0.0227682496375	2.36058111914329e-07\\
};
\addlegendentry{$\mathcal{O}(h^3)$}

\addplot [color=mycolor3, line width=1.0pt, mark=x, mark options={solid, mycolor3}]
  table[row sep=crcr]{%
0.344185326777365	0.0105517881375063\\
0.262247567979484	0.00602757720099625\\
0.183511800711609	0.00234563302038345\\
0.134653571411827	0.0013287408768638\\
0.0995855310120098	0.000514076605648044\\
0.0735033052111926	0.000239662234269506\\
0.054782539201058	0.000124568458529948\\
0.0408424807519244	7.49186548716811e-05\\
0.0304951795393483	3.085972622093e-05\\
0.0227682496375	1.55903754069696e-05\\
};
\addlegendentry{$L^\infty(H^1)$-error}

\addplot [color=mycolor4, line width=1.0pt, mark=o, mark options={solid, mycolor4}]
  table[row sep=crcr]{%
0.344185326777365	0.000758002279879271\\
0.262247567979484	0.000291258758732939\\
0.183511800711609	8.10612904705798e-05\\
0.134653571411827	3.46496717020915e-05\\
0.0995855310120098	1.09677308611865e-05\\
0.0735033052111926	3.68133291320008e-06\\
0.054782539201058	1.4021227598249e-06\\
0.0408424807519244	5.833082788723e-07\\
0.0304951795393483	1.99339703783228e-07\\
0.0227682496375	1.41814427162488e-07\\
};
\addlegendentry{$L^\infty(L^2)$-error}

\end{axis}

\begin{axis}[%
width=1.813in,
height=2.971in,
at={(3.089in,0.401in)},
scale only axis,
xmode=log,
xmin=0.0227682496375,
xmax=0.344185326777365,
xminorticks=true,
xlabel style={font=\color{white!15!black}},
xlabel={mesh size ($h$)},
ymode=log,
ymin=1e-08,
ymax=1e-01,
yminorticks=true,
ylabel style={font=\color{white!15!black}},
ylabel={error},
axis background/.style={fill=white},
title={Error in $X$},
legend style={at={(0.97,0.03)}, anchor=south east, legend cell align=left, align=left, draw=white!15!black}
]
\addplot [color=mycolor1, dashdotted, line width=1.0pt]
  table[row sep=crcr]{%
0.344185326777365	0.0118463539168842\\
0.262247567979484	0.00687737869111541\\
0.183511800711609	0.00336765810004173\\
0.134653571411827	0.001813158429396\\
0.0995855310120098	0.000991727798694396\\
0.0735033052111926	0.000540273587696974\\
0.054782539201058	0.000300112660131545\\
0.0408424807519244	0.000166810823397132\\
0.0304951795393483	9.29955975137088e-05\\
0.0227682496375	5.1839319155552e-05\\
};
\addlegendentry{$\mathcal{O}(h^2)$}

\addplot [color=mycolor2, dashed, line width=1.0pt]
  table[row sep=crcr]{%
0.344185326777365	0.000815468238800619\\
0.262247567979484	0.000360715167163788\\
0.183511800711609	0.000123601000423939\\
0.134653571411827	4.8829651610726e-05\\
0.0995855310120098	1.97523478904706e-05\\
0.0735033052111926	7.94237888280735e-06\\
0.054782539201058	3.28818671367803e-06\\
0.0408424807519244	1.362593568762e-06\\
0.0304951795393483	5.67183488509904e-07\\
0.0227682496375	2.36058111914329e-07\\
};
\addlegendentry{$\mathcal{O}(h^3)$}

\addplot [color=mycolor3, line width=1.0pt, mark=x, mark options={solid, mycolor3}]
  table[row sep=crcr]{%
0.344185326777365	0.00663685738886193\\
0.262247567979484	0.00390611494444339\\
0.183511800711609	0.00150931977007308\\
0.134653571411827	0.000871292875549033\\
0.0995855310120098	0.000328478105622272\\
0.0735033052111926	0.000156631356985093\\
0.054782539201058	8.16040886399665e-05\\
0.0408424807519244	4.95580507007103e-05\\
0.0304951795393483	2.00737632581705e-05\\
0.0227682496375	1.02685126798088e-05\\
};
\addlegendentry{$L^\infty(H^1)$-error}

\addplot [color=mycolor4, line width=1.0pt, mark=o, mark options={solid, mycolor4}]
  table[row sep=crcr]{%
0.344185326777365	0.000414992523310591\\
0.262247567979484	0.000164727747573887\\
0.183511800711609	4.53815793555502e-05\\
0.134653571411827	1.98370486779085e-05\\
0.0995855310120098	6.32398302536508e-06\\
0.0735033052111926	2.16095716560302e-06\\
0.054782539201058	8.07898846813942e-07\\
0.0408424807519244	3.38224190193256e-07\\
0.0304951795393483	1.14811512294822e-07\\
0.0227682496375	4.2741549713073e-08\\
};
\addlegendentry{$L^\infty(L^2)$-error}

\end{axis}
\end{tikzpicture}%
	\end{center}
	\caption{Convergence rate of the evolving quadratic finite element discretization of Example~\ref{subsection: example 1}.}\label{fig: ex1_bdf3_spatial}
\end{figure}

\begin{remark}[Linear finite elements]
	We solved the same problem with linear finite elements. Although not covered by the theory of this paper, we observe \brev the expected \erev $\mathcal{O}(h^2)$-convergence in $L^2$-norm and $\mathcal{O}(h)$-convergence in $H^1$-norm.
\end{remark}

\brev
\subsection{An evolving 3d domain}\label{ex2}
This example is similar to the previous one, but in three dimensions. We consider \eqref{eq: Poisson equation velocity} for $t \in [0,0.1]$, with $\Om(0)$ being the unit ball in $\R^3$. As exact solution, we choose 
\begin{align}
	v(x,t) = \begin{pmatrix}
		\sin(-t/10)(x_1^2-2x_2^2+x_3^2) \\
		\sin(-t/10)(\exp(x_2)\sin(x_3)-\exp(x_3)\sin(x_2)) \\
		\exp(-5t)(x_1^2-x_3^2)
	\end{pmatrix}
\end{align}
which satisfies $-\Delta v = 0$. We use isoparametric finite elements of second order. For the time discretization and reference solution, we proceed as in the previous example, with $\tau = 10^{-3}$ for the BDF method and $\tau = 10^{-6}$ for the reference solution. We record the $L^\infty(L^2)$- and $L^\infty(H^1)$-norm of the position and velocity error. The results are shown in Figure~\ref{fig: ex2_convergence}. 
\erev
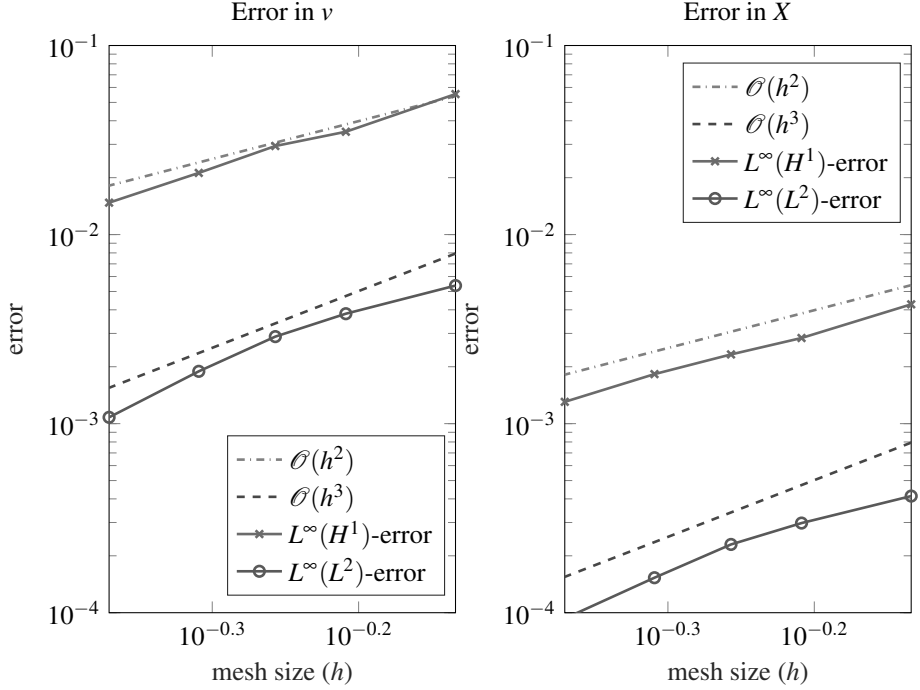
\begin{figure}[htb!]
	\begin{center}
%
%
\definecolor{mycolor1}{rgb}{0.5,0.5,0.5}%
\definecolor{mycolor2}{rgb}{0.28,0.28,0.28}%
\definecolor{mycolor3}{rgb}{0.42,0.42,0.42}%
\definecolor{mycolor4}{rgb}{0.35,0.35,0.35}%
\begin{tikzpicture}

\begin{axis}[%
width=1.813in,
height=2.971in,
at={(0.704in,0.401in)},
scale only axis,
xmode=log,
xmin=0.426113479553984,
xmax=0.734922715464332,
xminorticks=true,
xlabel style={font=\color{white!15!black}},
xlabel={mesh size ($h$)},
ymode=log,
ymin=1e-04,
ymax=0.1,
yminorticks=true,
ylabel style={font=\color{white!15!black}},
ylabel={error},
axis background/.style={fill=white},
title={Error in $v$},
legend style={at={(0.97,0.03)}, anchor=south east, legend cell align=left, align=left, draw=white!15!black}
]
\addplot [color=mycolor1, dashdotted, line width=1.0pt]
  table[row sep=crcr]{%
0.734922715464332	0.0540111397705467\\
0.618455933616664	0.038248774182566\\
0.553472996485489	0.0306332357838626\\
0.490672698071972	0.0240759696633229\\
0.426113479553984	0.0181572697457604\\
};
\addlegendentry{$\mathcal{O}(h^2)$}

\addplot [color=mycolor2, dashed, line width=1.0pt]
  table[row sep=crcr]{%
0.734922715464332	0.00793880270109875\\
0.618455933616664	0.00473103626935436\\
0.553472996485489	0.00339093376026819\\
0.490672698071972	0.00236268419868032\\
0.426113479553984	0.00154741147811325\\
};
\addlegendentry{$\mathcal{O}(h^3)$}

\addplot [color=mycolor3, line width=1.0pt, mark=x, mark options={solid, mycolor3}]
  table[row sep=crcr]{%
0.734922715464332	0.0553454125043073\\
0.618455933616664	0.0349904550602355\\
0.553472996485489	0.0294095353824828\\
0.490672698071972	0.0211932812561844\\
0.426113479553984	0.0147509758715837\\
};
\addlegendentry{$L^\infty(H^1)$-error}

\addplot [color=mycolor4, line width=1.0pt, mark=o, mark options={solid, mycolor4}]
  table[row sep=crcr]{%
0.734922715464332	0.00537280287934858\\
0.618455933616664	0.0038148991275289\\
0.553472996485489	0.0028860960388423\\
0.490672698071972	0.00189140455742784\\
0.426113479553984	0.00107990040671304\\
};
\addlegendentry{$L^\infty(L^2)$-error}

\end{axis}

\begin{axis}[%
width=1.813in,
height=2.971in,
at={(3.089in,0.401in)},
scale only axis,
xmode=log,
xmin=0.426113479553984,
xmax=0.734922715464332,
xminorticks=true,
xlabel style={font=\color{white!15!black}},
xlabel={mesh size ($h$)},
ymode=log,
ymin=1e-04,
ymax=0.1,
yminorticks=true,
ylabel style={font=\color{white!15!black}},
ylabel={error},
axis background/.style={fill=white},
title={Error in $X$},
legend style={at={(0.97,0.97)}, anchor=north east, legend cell align=left, align=left, draw=white!15!black}
]
\addplot [color=mycolor1, dashdotted, line width=1.0pt]
  table[row sep=crcr]{%
0.734922715464332	0.00540111397705467\\
0.618455933616664	0.0038248774182566\\
0.553472996485489	0.00306332357838626\\
0.490672698071972	0.00240759696633229\\
0.426113479553984	0.00181572697457604\\
};
\addlegendentry{$\mathcal{O}(h^2)$}

\addplot [color=mycolor2, dashed, line width=1.0pt]
  table[row sep=crcr]{%
0.734922715464332	0.000793880270109875\\
0.618455933616664	0.000473103626935436\\
0.553472996485489	0.000339093376026819\\
0.490672698071972	0.000236268419868032\\
0.426113479553984	0.000154741147811325\\
};
\addlegendentry{$\mathcal{O}(h^3)$}

\addplot [color=mycolor3, line width=1.0pt, mark=x, mark options={solid, mycolor3}]
  table[row sep=crcr]{%
0.734922715464332	0.00427537877470703\\
0.618455933616664	0.00283764525660276\\
0.553472996485489	0.00232307318149603\\
0.490672698071972	0.00182810458550601\\
0.426113479553984	0.00130407481765572\\
};
\addlegendentry{$L^\infty(H^1)$-error}

\addplot [color=mycolor4, line width=1.0pt, mark=o, mark options={solid, mycolor4}]
  table[row sep=crcr]{%
0.734922715464332	0.00041390709406179\\
0.618455933616664	0.000297950867494692\\
0.553472996485489	0.000229645000020947\\
0.490672698071972	0.000153130362551234\\
0.426113479553984	9.31737815971896e-05\\
};
\addlegendentry{$L^\infty(L^2)$-error}

\end{axis}
\end{tikzpicture}%
	\end{center}
	\caption{Convergence rate of the evolving quadratic finite element discretization of Example~\ref{ex2}.}\label{fig: ex2_convergence}
\end{figure}

\subsection{Diffusion equation}\label{subsection: example 2}
In this example, we consider the diffusion equation \eqref{eq: conservation and diffusion law}, where the velocity again satisfies \eqref{eq: Poisson equation velocity}. As exact solution, we choose $\beta = 1$ and
\begin{align}
	u(x,y,t) &= e^{-t}(x^2+y^2)(x^2-y^2) \,,\\
	v(x,y,t) &= \left( 1 - \frac{1}{r(t)} \right) \begin{pmatrix} x \\ y \end{pmatrix} + \begin{pmatrix} -y \\ x \end{pmatrix} \,, \quad\text{where}\quad r(t) = \frac{2}{1+e^{-t}} \,.
\end{align}
The velocity $v$ describes a growing ball which in addition is rotating anti-clockwise (cf. \cite[Example 11.1]{KLLP17}), $r(t)$ is the radius of the ball at $t \in [0,T]$. We compute the right-hand side functions $f$ and $g$ of \eqref{eq: conservation and diffusion law} and apply second order isoparametric finite elements in space and a linearly implicit 4~step BDF method \brev with time step-size $\tau=10^{-3}$ \erev in time.

Note that $v$ is linear in $x$ and $y$, so the solution to $-\laplace v = 0$ is computed exactly by the finite element method. This is reflected in the convergence plot in $v$, which shows a purely temporal convergence and is thus not shown here. \brev We record the error
\begin{align}
	\| \mathrm{err}_\bfu \|_{L^\infty(L^2)} &:= \sup_{n: t_n \le 1} \| (u_h^n)^L - u(\cdot,t_n) \|_{L^2(\Om(t_n))} \,,\\
	\| \mathrm{err}_\bfu \|_{L^2(H^1)} &:= \left( \tau \sum_{n: t_n \le 1} \| (u_h^n)^L - u(\cdot,t_n) \|_{H^1(\Om(t_n))}^2 \right)^\frac{1}{2}  \,,
\end{align}
where $\tau$ denotes the time step size and $t_n = n \tau$ the $n$-th time step. Figure~\ref{fig: ex2_bdf3_spatial} shows the results. As expected, the error in the $L^2(H^1)$-norm converges with the expected order, whereas the $L^2$-norm convergence rate is one order higher. \erev
\begin{figure}[htb!]
	\begin{center}
%
%
\definecolor{mycolor1}{rgb}{0.5,0.5,0.5}%
\definecolor{mycolor2}{rgb}{0.28,0.28,0.28}%
\definecolor{mycolor3}{rgb}{0.42,0.42,0.42}%
\definecolor{mycolor4}{rgb}{0.35,0.35,0.35}%
\begin{tikzpicture}

\begin{axis}[%
width=4.198in,
height=2.944in,
at={(0.704in,0.429in)},
scale only axis,
xmode=log,
xmin=0.0408424807519244,
xmax=0.344185326777365,
xminorticks=true,
xlabel style={font=\color{white!15!black}},
xlabel={mesh size ($h$)},
ymode=log,
ymin=1e-06,
ymax=0.1,
yminorticks=true,
ylabel style={font=\color{white!15!black}},
ylabel={error},
axis background/.style={fill=white},
title={Error in $u$},
legend style={at={(0.97,0.03)}, anchor=south east, legend cell align=left, align=left, draw=white!15!black}
]
\addplot [color=mycolor1, dashdotted, line width=1.0pt]
  table[row sep=crcr]{%
0.344185326777365	0.0592317695844208\\
0.262247567979484	0.034386893455577\\
0.183511800711609	0.0168382905002086\\
0.134653571411827	0.00906579214697999\\
0.0995855310120098	0.00495863899347198\\
0.0735033052111926	0.00270136793848487\\
0.054782539201058	0.00150056330065773\\
0.0408424807519244	0.000834054116985658\\
};
\addlegendentry{$\mathcal{O}(h^2)$}

\addplot [color=mycolor2, dashed, line width=1.0pt]
  table[row sep=crcr]{%
0.344185326777365	0.00815468238800619\\
0.262247567979484	0.00360715167163788\\
0.183511800711609	0.00123601000423939\\
0.134653571411827	0.00048829651610726\\
0.0995855310120098	0.000197523478904706\\
0.0735033052111926	7.94237888280735e-05\\
0.054782539201058	3.28818671367803e-05\\
0.0408424807519244	1.362593568762e-05\\
};
\addlegendentry{$\mathcal{O}(h^3)$}

\addplot [color=mycolor3, line width=1.0pt, mark=x, mark options={solid, mycolor3}]
  table[row sep=crcr]{%
0.344185326777365	0.051599555095687\\
0.262247567979484	0.0239321477625058\\
0.183511800711609	0.0115393709707093\\
0.134653571411827	0.00513095735207886\\
0.0995855310120098	0.00215365052522033\\
0.0735033052111926	0.00104272912364416\\
0.054782539201058	0.000498671613573675\\
0.0408424807519244	0.000244644110088065\\
};
\addlegendentry{$L^2(H^1)$-error}

\addplot [color=mycolor4, line width=1.0pt, mark=o, mark options={solid, mycolor4}]
  table[row sep=crcr]{%
0.344185326777365	0.00531945267777078\\
0.262247567979484	0.0016778718154684\\
0.183511800711609	0.000601409365413919\\
0.134653571411827	0.000213714266863887\\
0.0995855310120098	7.76555461076461e-05\\
0.0735033052111926	2.72093148718381e-05\\
0.054782539201058	9.79138597375079e-06\\
0.0408424807519244	3.50208583971646e-06\\
};
\addlegendentry{$L^\infty(L^2)$-error}

\end{axis}
\end{tikzpicture}%
	\end{center}
	\caption{Convergence rate in $u$ of the evolving quadratic finite element discretization of Example~\ref{subsection: example 2}.}\label{fig: ex2_bdf3_spatial}
\end{figure}
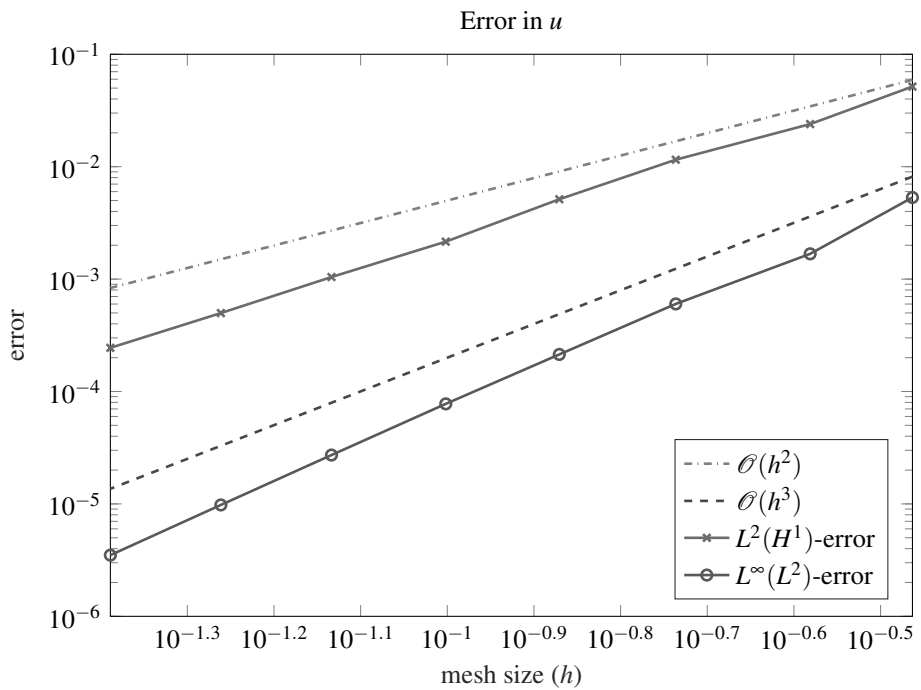

%
\section*{Acknowledgments}
The author is very grateful to Christian Lubich and Bal\'{a}zs Kov\'{a}cs for stimulating discussions on the topic and their help during the work on this manuscript. \brev We thank the anonymous referees for helpful and constructive comments. \erev

\bibliography{biblio}{}
\bibliographystyle{authordate1}
\end{document}